\documentclass{amsart}

\usepackage{hyperref,amssymb,graphicx,subcaption}

\def\Riem{\mathop{\rm Rm}}

\newtheorem{theorem}{Theorem}[section]
\newtheorem{proposition}[theorem]{Proposition}
\newtheorem{lemma}[theorem]{Lemma}
\newtheorem{corollary}[theorem]{Corollary}
\newtheorem{defn}{Definition}

\title{Asymptotic geometry of toric K\"ahler instantons}

\author{Brian Weber}

\begin{document}
	
	\maketitle
	
	\begin{abstract}
		The symplectic reduction of a complete toric K\"ahler manifold need not be closed or be a polygon.
		In the scalar-flat K\"ahler case we establish natural geometric criteria for these reductions to be closed, and classify the asymptotic geometries that may occur.
	\end{abstract}

	\section{Introduction}
	
	A K\"ahler manifold $(M^4,J,g)$ with a 2-torus action that preserves the symplectic structure and metric is said to be a toric K\"ahler 4-manifold.
	If the manifold is compact, the quotient is a compact Delzant polygon (see \cite{Delzant} \cite{SymplMS}), but if the manifold is only complete, its quotient need not be closed and need not be a polygon; see the example in Section \ref{SubsubsecNonpolygon}.
	We give necessary and sufficient conditions for the quotient to be closed, and when $M^4$ is scalar-flat we classify its possible metrics.
	
	Throughout, $M^4=(M^4,J,g,\mathcal{X}^1,\mathcal{X}^2)$ will be a complete K\"ahler 4-manifold where $\mathcal{X}^1$, $\mathcal{X}^2$ are two symplectomorphic Killing fields that commute.
	This means $\mathcal{L}_{\mathcal{X}^i}\omega=\mathcal{L}_{\mathcal{X}^i}g=\mathcal{L}_{\mathcal{X}^i}J=0$, and $[\mathcal{X}^1,\mathcal{X}^2]=0$.
	The fields $\mathcal{X}^1$, $\mathcal{X}^2$ may be generators of a torus action, but dealing with Killing fields rather than action tori allows additional flexibility, not least because we can take arbitrary linear combination of the fields without worrying about whether it may come from a torus automorphism.
	It is convenient to assume $M^4$ is simply connected; we may pass to the universal cover if not.
	Because $\mathcal{L}_{\mathcal{X}^i}\omega=0$ we have $di_{\mathcal{X}^i}\omega=0$, and therefore functions $\varphi^1$, $\varphi^2$ exist with $\omega(\mathcal{X}^i,\cdot)=-d\varphi^i$.
	These are called the \textit{momentum} or \textit{action} coordinates on $M^4$.
	Then $\nabla\varphi^1=-J\mathcal{X}^1$, $\nabla\varphi^2=-J\mathcal{X}^2$, and the usual Nijenhuis relation implies that also $[\nabla\varphi^1,\nabla\varphi^2]=0$.
	To complete $(\varphi^1,\varphi^2)$ into a coordinate system we first choose a single leaf of the $\nabla\varphi^1$-$\nabla\varphi^2$ distribution, assign it coordinates $(\theta_1,\theta_2)=(0,0)$, and then push $\theta_1$, $\theta_2$ along the action created by the $\mathcal{X}^1$-$\mathcal{X}^2$ fields; these are the \textit{cyclic} or \textit{angle} coordinates.
	The coordinates $(\varphi^1,\theta_1,\varphi^2,\theta_2)$ are known as action-angle coordinates; see \cite{Ar}.
	The Arnold-Liouville reduction is
	\begin{equation}
		\Phi:M^4\rightarrow\mathbb{R}^2, \quad\quad
		\Phi(p)=\left(\varphi^1(p),\varphi^2(p)\right).
	\end{equation}
	The set $\Sigma^2=\Phi(M^4)$ in the $\varphi^1$-$\varphi^2$ plane is called the reduction of $M^4$.
	Because $g$ is invariant the reduction is generically a Riemannian submersion and $Int(\Sigma^2)$ inherits a Riemannian metric, $g_\Sigma$.
	We call $(\Sigma^2,g_\Sigma)$ the \textit{metric reduction} of $M^4$.
	
	Although $\Sigma^2$ need not be closed nor a polygon, we recover convexity.
	\begin{proposition}[\textit{cf.} Proposition \ref{PropConvex}] \label{PropConvexity}
		Assume $(M^4,J,g,\mathcal{X}^1,\mathcal{X}^2)$ is geodesically complete.
		Then its reduction $\Sigma^2$ is convex in the $\varphi^1$-$\varphi^2$ plane.
	\end{proposition}
	We use the term ``boundary point'' of $\Sigma^2$ to mean the extrinsic boundary of $\Sigma^2$ in the topology of the coordinate plane (as opposed to the intrinsic metric topology on $(\Sigma^2,g_{\Sigma})$), and we use $\partial\Sigma^2$ to indicate the set of all such points in the $\varphi^1$-$\varphi^2$ plane.
	We define the ``included boundary points'' to be $(\partial\Sigma^2){}_I=\partial\Sigma^2\cap\Sigma^2$ and the ``non-included boundary points'' to be $(\partial\Sigma^2){}_N=\partial\Sigma^2\setminus\Sigma^2$.
	If $p$ is a point in the $\varphi^1$-$\varphi^2$ coordinate plane and $\epsilon>0$, we use the notation $D_p(\epsilon)$ to mean the coordinate disk of radius $\epsilon$ around $p$; then $p\in\partial\Sigma^2$ if and only if every $D_p(\epsilon)$ contains both points of $\Sigma^2$ and $\mathbb{R}^2\setminus\Sigma^2$.
	See Figure~\ref{FigInitial}.
	\begin{figure}[h]
		\begin{tabular}{lll}
			\includegraphics[scale=0.4]{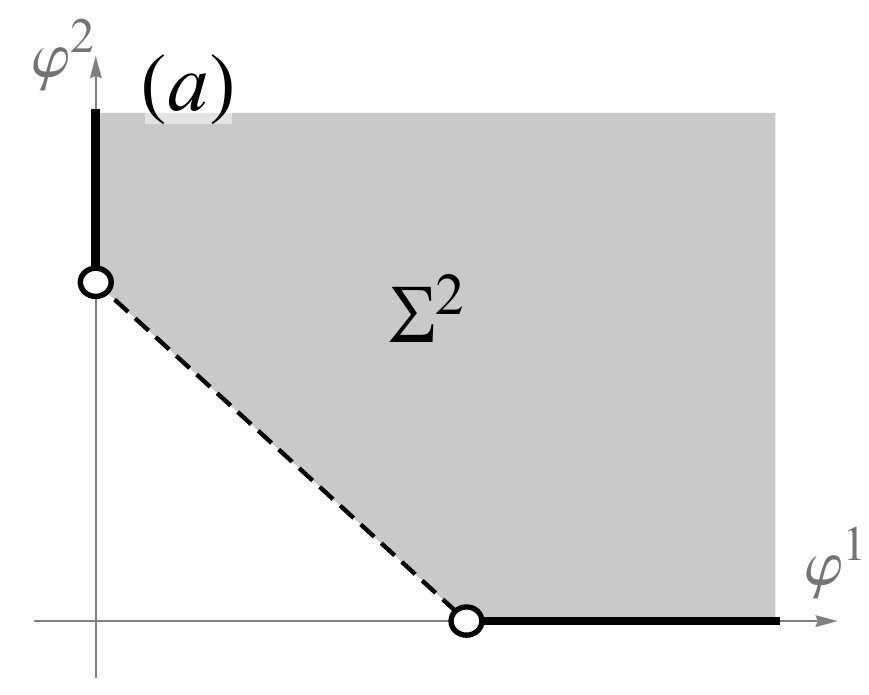}
			& \hspace{-0.0in}\includegraphics[scale=0.4]{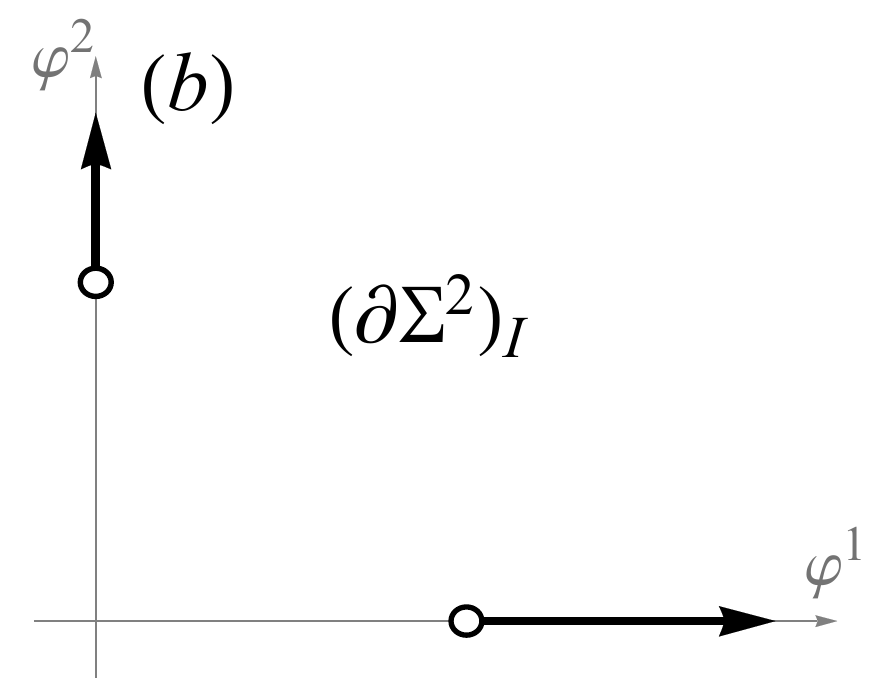}
			& \hspace{-0.0in}\includegraphics[scale=0.4]{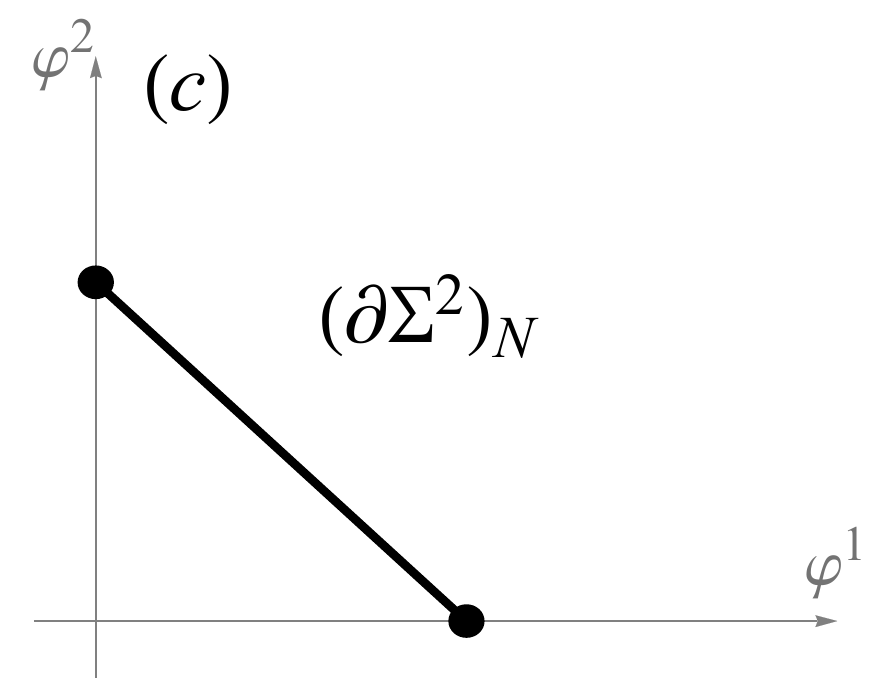}
		\end{tabular}
		\caption{
			(a) A typical reduction showing included and non-included boundary points with (b) the included boundary points and (c) the non-included boundary points.
			%The reduction itself is the ``modified Taub-NUT of the first kind'' of \cite{NW}, or the $n=2$ metric of \cite{FYZ}.
			%It is scalar-flat and 2-ended, with an ALE end and a cusp-like end.
		}
		\label{FigInitial}
	\end{figure}
	%Proposition \ref{PropBdStruct1} states that an \textit{included} boundary point is on an edge or a vertex.
	\begin{proposition}[Structure of $(\partial\Sigma^2)_I$] \label{PropBdStruct1}
		The included boundary points $p\in\Sigma^2\cap\partial\Sigma^2$ are precisely those for which $\Phi^{-1}(p)\subset{}M^4$ consists of points where the distribution $\{\mathcal{X}^1,\mathcal{X}^2\}$ has rank 1 or 0.
		In this case, a coordinate disk $D_p(\epsilon)$, $\epsilon>0$ exists so that, if the rank is $1$ at $\Phi^{-1}(p)$, then a linear function $m(\varphi^1,\varphi^2)$ exists so that $\Sigma^2\cap{}D_p(\epsilon)=\{m\ge0\}\cap{}D_p(\epsilon)$, and the rank is 0 then two linear functions $m_1$, $m_2$ exist so that $\Sigma^2\cap{}D_p(\epsilon)=\{m_1\ge0\}\cap\{m_2\ge0\}\cap{}D_p(\epsilon)$.
	\end{proposition}
	\begin{proposition}[Structure of $(\partial\Sigma^2)_N$] \label{PropBdStruct2}
		If $q\in\partial\Sigma^2\setminus\Sigma^2$ is a non-included boundary point, then it is infinitely far away in the sense that if $\{p_i\}\subset{}M^4$ is any sequence of points with $\Phi(p_i)\rightarrow{}q$ in the coordinate topology, then $p_i$ diverges in $M^4$.
	\end{proposition}
	By Proposition \ref{PropBdStruct1}, the included boundary $(\partial\Sigma^2)_I$ consists of segments, rays, and lines in the $\varphi^1$-$\varphi^2$ plane, joined at vertex points.
	Proposition \ref{PropBdStruct3} says that such a segment, ray, or line $l\in(\partial\Sigma^2)_I$, is the image of the zero-set of a Killing field.
	When $l$ is maximally extended, we call its preimage $L^2=\Phi(l)$ a \textit{polar submanifold}---this is an analogy with the fact that the zero-set of a rotational Killing field on $\mathbb{S}^2$ consists of its north and south pole.
	See Figure \ref{FigDepictL2} for a depiction.
	\begin{proposition}[Polar submanifolds] \label{PropBdStruct3}
		Assume $l\subset(\partial\Sigma^2)_I$ is a maximally extended boundary segment, ray, or line that has one or both of its termini at included vertex points or non-included points.
		Then $L^2=\Phi^{-1}(l)$ is a holomorphically embedded, geodesically complete codimension-2 submanifold with a Killing field $\mathcal{X}$.
		
		Any terminus of $l$ at a non-included point corresponds to a manifold end of $L^2$.
		A terminus of $l$ at an included vertex point corresponds to a zero of the Killing field on $L^2$.
	\end{proposition}
	\begin{figure}[h]
		\includegraphics[scale=0.4]{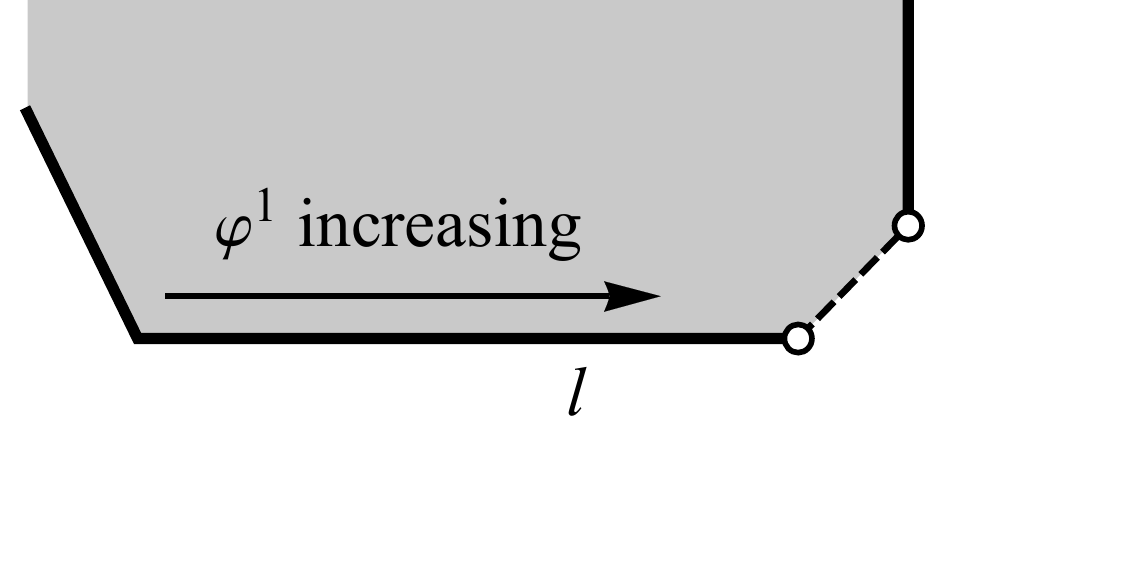}\hspace{0.25in}	\includegraphics[scale=0.4]{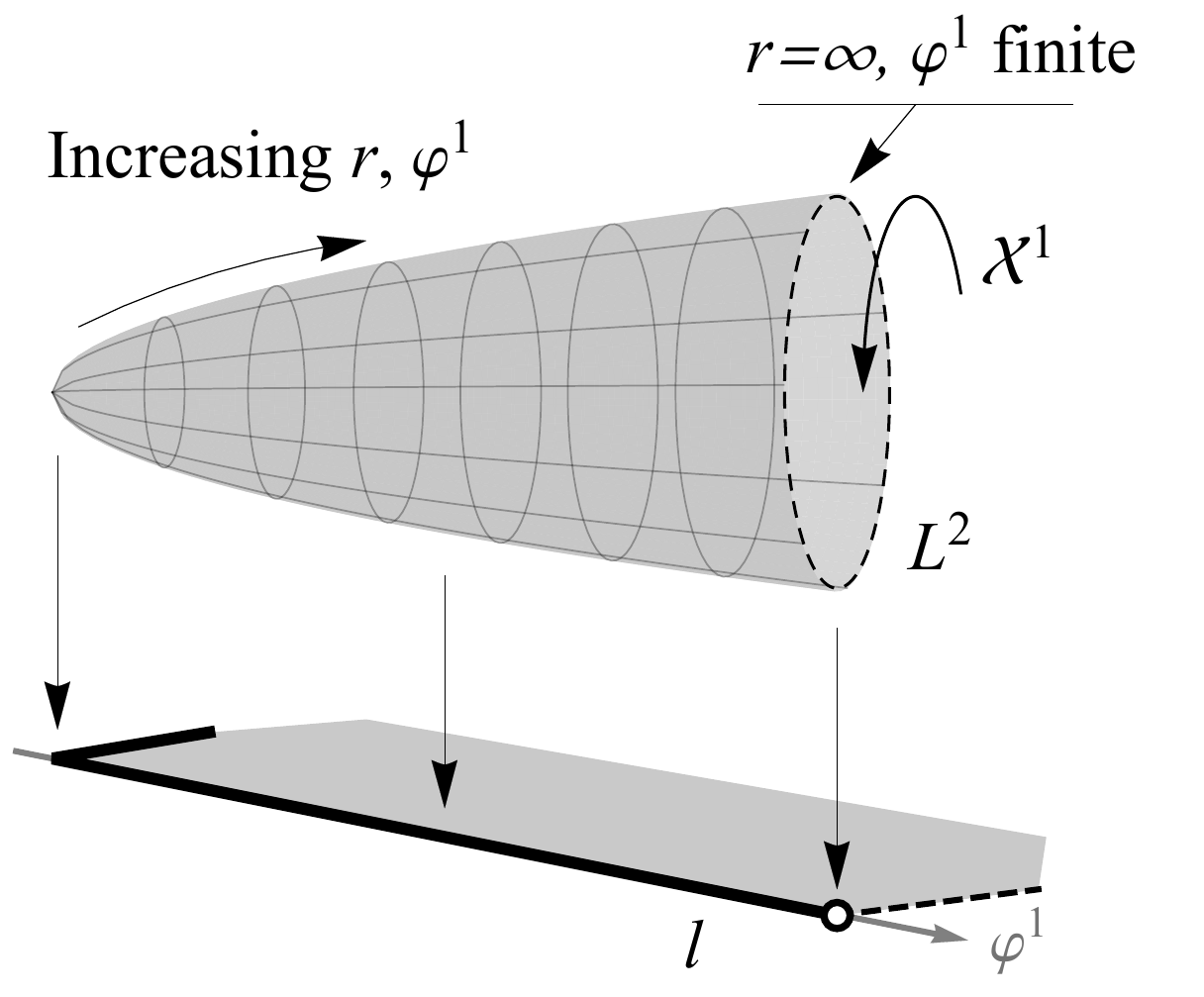}
		\caption{
			A polar submanifold $L^2\subset{}M^2$ over a segment $l\subset\Sigma^2$, with distance function $r$ and momentum function $\varphi^1$.
			Depicted is the possibility that $l$ has a terminus on a non-included point.
		}
		\label{FigDepictL2}
	\end{figure}
	We give geometric criteria on $M^4$ that forces the reduction $\Sigma^2$ to be closed.
	First we remark that the Killing field $\mathcal{X}$ on a polar submanifold $L^2$ gives rise to a momentum function $\varphi:L^2\rightarrow\mathbb{R}$ in the usual way: $d\varphi=-i_{\mathcal{X}}\omega$.
	Then using $\varphi$ we have a standard construction for a distance function $r:L^2\rightarrow\mathbb{R}^2$, this being $dr=|d\varphi|^{-1}d\varphi$.
	The trajectories of $\nabla{}r$ are perpendicular to the Killing field $\mathcal{X}$, and we call such a distance function a radial distance function.
	\begin{theorem}[Criteria for closedness of $\Sigma^2$]\label{ThmClosed}
		Whenever $L^2\subset{}M^4$ is a polar submanifold that has an unbounded radial distance function $r$, assume one of the following holds:
		\begin{itemize}
			\item The Killing field $\mathcal{X}$ on $L^2$ decays slowly (or not at all): a constant $C_1>0$ exist so when $|r|$ is sufficiently large, then $|\mathcal{X}|\ge{}C_1|r|^{-1}$,
			\item The negative part of the Gaussian curvature $K$ of $L^2$ decays quickly: when $|r|$ is sufficiently large, then $K\ge-2/r^2$.
		\end{itemize}
		Then the image $l\subset\Sigma^2$ of any polar submanifold $L^2$ is closed, if $L^2$ has an unbounded distance function then $l$ has at most one terminal point (which is an included point).
		If every component of $\partial\Sigma^2$ has an included point, then $\Sigma^2$ is closed.
	\end{theorem}
	%Because $\Sigma^2$ is a convex subset of the plane, $\partial\Sigma^2$ must have zero, one, or two components.
	%The case of zero boundary components is that $\Sigma^2$ \textit{is} the plane.
	%Due to convexity, the case of two boundary components is that $\Sigma^2$ is a strip in the plane with each component of $\partial\Sigma^2$ being a line.
	
	%Without further conditions on a toric 4-manifold, Theorem \ref{ThmClosed} is as far as we can go.
	%If $M^4$ is scalar-flat, combining Theorem \ref{ThmClosed} with the classification of \cite{Web1} allows us to classify all manifolds $M^4$ whose reductions $\Sigma^2$ are closed.
	We address the case that $(\partial\Sigma^2)_I$ is empty.
	\begin{theorem}[c.f. Theorem \ref{ThmRTwoClassification}] \label{ThmComplete}
		Assume $(M^4,J,g,\mathcal{X}^1,\mathcal{X}^2)$ has non-negative scalar curvature $s\ge0$ and that the distribution $\{\mathcal{X}^1,\mathcal{X}^2\}$ is always rank 2 (this is the same as $M^4\rightarrow\Sigma^2$ being a Riemannian submersion, or $(\partial\Sigma^2)_I=\varnothing$).
		
		Then $M^4$ is flat $\mathbb{C}^2$, its Killing fields are translations, and $(\Sigma^2,g_\Sigma)$ is flat $\mathbb{R}^2$.
	\end{theorem}
	Finally we consider the scalar-flat case, and classify these metrics based on their natural asymptotic characteristics.
	We require a broader range of asymptotic models that the usual ALE-ALF-ALG-ALH schema of \cite{CK} \cite{Etesi}.
	The ALF asymptotic model in particular is much too rigid, due to the fact that it requires not only quadratic curvature decay and cubic volume growth, but a very particular collapsing condition: the Hopf circles on the $\mathbb{S}^3$ level-sets of the distance function must asymptotically be the fibers of a collapsing F-structure (in the sense of Cheeger-Gromov \cite{CG86}).
	This condition is too inflexible to model the needed range of phenomena, as the examples of the generalized Taub-NUT metrics demonstrates \cite{Web3}.
	
	To create the needed range of asymptotic models, we use the following definition.
	\begin{defn} \label{DefModeled}
		Assume $(M^n,g)$ and $(\widetilde{M}{}^n,\tilde{g})$ are Riemannian manifolds and that a compact set $\widetilde{K}\subset\widetilde{M}{}^n$ and covering map $\pi:\widetilde{M}{}^4\setminus\widetilde{K}\rightarrow{}M^4$ exist so that
		\begin{equation}
			\left\|\pi^*g\,-\,\tilde{g}\right\|_{\tilde{g}}\;=\;O(\tilde{\rho}{}^{-l})
			\quad \text{for $l>0$, or $o(1)$ for $l=0$}
		\end{equation}
		where $\tilde\rho$ is a distance function on $\widetilde{M}^4$.
		Then we say $(M^n,g)$ is \emph{asymptotically modeled on $(\widetilde{M}{}^n,\tilde{g})$} to degree $l$.
		If, in addition, the $M^n$ has Killing fields $\{\mathcal{X}^i\}$ and the model $\widetilde{M}{}^4$ has Killing fields $\{\widetilde{\mathcal{X}}{}^i\}$ so that $\pi_*\widetilde{\mathcal{X}}{}^i=\mathcal{X}^i$, we say $M^4$ is \emph{equivariantly asymptotically modeled} on $\widetilde{M}^4$ to degree $l$.
	\end{defn}
	This is similar to typical definitions in the literature except we do not require that the manifold's curvature be similar to the model curvature.
	
	\begin{figure}[h]
		\begin{tabular}{ll}
			\includegraphics[scale=0.35]{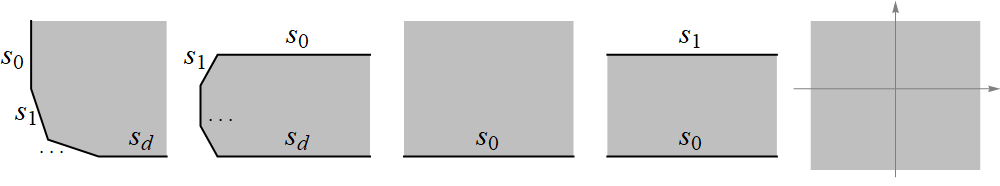}
		\end{tabular}
		\caption{
			The five closed polygon types: the general case, parallel-ray case, half-plane case, parallel-line case, and the plane.
		} \label{FigFiveModel}
	\end{figure}
	Depicted in Figure~\ref{FigFiveModel} are the five kinds of non-compact closed polygons.
	
	From \cite{Web1}, we know any manifold $M^4$ whose reduction is a half-plane has only has a single nontrivial scalar-flat metric up to homothety: the ``exceptional half-plane'' instanton metric.
	The fourth case, the strip, is not well studied.
	The fifth case, the plane, is always flat by Theorem \ref{ThmComplete}.
	The general case and parallel-ray case give larger varieties of metrics, all of which were described explicitly in the classification theorem of \cite{Web1}.
	
	%The following theorem gives necessary and sufficient conditions that a complete scalar-flat $M^4$ has a closed reduction $\Sigma^2$.
	\begin{theorem} \label{ThmSFClassification}
		Assume $(M^4,J,g,\mathcal{X}^1,\mathcal{X}^2)$ is a toric K\"ahler 4-manifold that is scalar-flat, geodesically complete, and has finite topological type.
		
		Then its reduction $\Sigma^2$ is closed if and only if $M^4$ is one of the following:
		\begin{enumerate}
			\item $M^4$ is asymptotically ``general,'' meaning it is one of the following:
			\begin{itemize}
				\item \emph{``Asymptotically ALE''}, modeled on Euclidean space to order $1$
				\item \emph{``Asymptotically ALF''}, modeled on the Taub-NUT to order $1$
				\item \emph{``Asymptotically ALF-like''}, modeled on a generalized Taub-NUT with chirality $k\in(-1,0)\cup(0,1)$, to order $1$
				\item \emph{``Asymptotically Exceptional''}, modeled on a maximally chiral Taub-NUT (with chirality $k=-1$ or $k=1$), to order $1$
			\end{itemize}
			\item $M^4$ is ``asymptotically equivariantly $\mathbb{R}^2\times\mathbb{S}^2$,'' meaning its asymptotic model is any of the model metrics in (\ref{EqnSRCompMomentums}) of Section \ref{SubSecRS}, to order $0$,
			\item $M^4$ is an exceptional half-plane instanton of \cite{Web1}
			\item $\Sigma^2$ is a closed strip
			\item $M^4$ is flat
		\end{enumerate}
	\end{theorem}
	To define the terms in Theorem \ref{ThmSFClassification}, a metric is ``asymptotically ALE'' if it is modeled on Euclidean space in the sense of Definition \ref{DefModeled} (this differs from the standard definition of ``ALE'' in that we do not require a curvature decay condition).
	The ``asymptotically ALF'' metrics are modeled on the Taub-NUT metric of the classic literature \cite{Haw77}.
	The ``asymptotically ALF-like'' metrics are modeled on the family of generalized Taub-NUT metrics on $\mathbb{C}^2$ discovered by Donaldson \cite{Do2}; in the language of \cite{Web3} they have ``chirality'' $k$ in $(-1,0)\cup(0,1)$.
	They have quadratic curvature decay $|\Riem|=O(r^{-2})$, bounded $L^2$ curvature energy, cubic volume growth, and a Killing field that is asymptotically bounded.
	They are not ALF, except in the chirality 0 case, which is the classic Taub-NUT.
	The ``asympotically exceptional'' metrics are modeled on the exceptional Taub-NUTs of \cite{Web3}, which have maximum chirality $k=-1$ or $k=1$ and are complete scalar-flat K\"ahler metrics on $\mathbb{C}^2$.
	However curvature does not decay, volume growth is quartic, and $L^2(|\text{Ric}|)=\infty$.
	The aymptotic models of case (1) metrics were studied carefully in \cite{Web3}, and the metrics themselves were explicitly written down in \cite{Web1}.	
	
	The case (2) metrics, which all have polygon reductions with parallel rays, were all written down and classified in \cite{Web1}.
	This being the case, to date the geometries of their asymptotic models not been studied (although in Section \ref{SecAsymptoticAnalysis} we characterize some aspects of their geometries).
	The models themselves are given explicitly in (\ref{EqnSRCompMomentums}) of Section \ref{SubSecRS}.
	At present we only mention that the family of models all have half-strip reductions with three edge and parallel rays, and are the simplest possible of the case (2) metrics; see Figure~\ref{FigRSModel}.
	
	Case (3) consists of just one non-trivial metric; see \cite{Web1}.
	Case (4), the closed strip, is the only case where the metrics have not been classified.
	In case (5), the flat case, is any metric whose reduction has no edges, and any metric with compact reduction (which in the scalar-flat toric case are always flat).
	Also in case (5) is flat $\mathbb{C}^2$ with one rotational and one translational field, which has a half-plane reduction, and flat $\mathbb{C}^2$ with two rotational fields, which has a quarter-plane reduction.
	
	In Section \ref{SecExamples} we give some examples.
	
	\textbf{Remark.}
	This paper is a companion to \cite{Web2} and \cite{Web1}.
	Using analytic results of \cite{Web2}, the paper \cite{Web1} classified the possible metrics on scalar-flat toric K\"ahler 4-manifolds whose reductions are closed, with the exception of the strip and plane (cases (4) and (5)).
	That paper begins with hypotheses on $\Sigma^2$ rather than $M^4$, the main extrinsic assumption being ``Hypothesis A,'' that the reduction $\Sigma^2$ is closed.
	In this paper we have laid out what intrinsic conditions on the parent manifold $M^4$ are equivalent to the extrinsic ``Hypothesis A'' of that paper.

	%
	%
	%
	%
	%
	% Section 2: Preliminaries and Notation
	%
	%
	%
	%
	%
	\section{Preliminaries and Notation} \label{SecPrelims}
	
	This section recalls the theory of scalar-flat toric K\"ahler manifolds, and sets notation for the paper.
	See \cite{G} \cite{Ab1} \cite{Do1} \cite{Do2} \cite{AS} \cite{CDG} and references therein.
	The companion paper \cite{Web1} has similar notation.
	Section \ref{SubSecConfChangeOfMetric} has a new result, but otherwise this section's material is well-known so we are brief.

	%
	%
	%
	% Subsection 2.1: Properties of the Reduction
	%
	%
	%
	\subsection{Properties of the reduction} \label{SubSecRedProps}
	
	From $J\nabla\varphi^i=\frac{\partial}{\partial\theta_i}$, in action-angle coordinates $(\varphi^1,\varphi^2,\theta_1,\theta_2)$ we can write
	\begin{equation}
		g\;=\;\left(\begin{array}{c|c}
			G_{ij} & 0 \\
			\hline
			0 & G^{ij}
		\end{array}\right), 
		\;
		J\;=\;\left(\begin{array}{c|c}
			0 & -G^{ij} \\
			\hline
			G_{ij} & 0
		\end{array}\right), 
		\;
		\omega\;=\;\left(\begin{array}{c|c}
			0 & -Id \\
			\hline
			Id\; & 0
		\end{array}\right). \label{EqnsGJOmegaM}
	\end{equation}
	where $G_{ij}$ and its inverse $G^{ij}$ are the $2\times2$ matrices $G_{ij}=\big<\frac{\partial}{\partial\varphi^i},\frac{\partial}{\partial\varphi^j}\big>\;=\;\left<\nabla\theta_i,\nabla\theta_j\right>$, $G^{ij}=\left<\nabla\varphi^i,\nabla\varphi^j\right>=\big<\frac{\partial}{\partial\theta_i},\frac{\partial}{\partial\theta_j}\big>$.
	We have $\mathcal{X}^1=\frac{\partial}{\partial\theta_1}$, $\mathcal{X}^2=\frac{\partial}{\partial\theta_2}$, and the generating functions $\theta_1$, $\theta_2$ of the $\mathcal{X}^1$-$\mathcal{X}^2$ action are pluriharmonic, meaning $d(Jd\theta_i)=0$; see Lemma 2.1 of \cite{Web1}.
	Therefore conjugate pluriharmonic functions exist, that we denote $\xi_1$, $\xi_2$, given by $d\xi_i=Jd\theta_i$.
	This gives a holomorphic chart $z_1=\xi_1+\sqrt{-1}\theta_1$, $z_2=\xi_2+\sqrt{-1}\theta_2$, with a coordinate frame and coframe given by
	\begin{eqnarray}
		\begin{aligned}
			\frac{\partial}{\partial{z}_1}
			&\;=\;\frac12\left(\nabla\varphi^1-\sqrt{-1}\,\mathcal{X}^1\right), \quad
			\frac{\partial}{\partial{z}_2}
			\;=\;\frac12\left(\nabla\varphi^2-\sqrt{-1}\,\mathcal{X}^2\right) \\
			dz_1&\;=\;d\xi_1 +\sqrt{-1}\,d\theta_1, \quad\quad\quad\;\;
			dz_2\;=\;d\xi_2+\sqrt{-1}\,d\theta_2.
		\end{aligned} \label{EqnCxCoords}
	\end{eqnarray}
	In this holomorphic frame the Hermitian metric is $h^{i\bar\jmath}=\big<\frac{\partial}{\partial{}z_i},\frac{\partial}{\partial{}\bar{z}_j}\big>=\frac12G^{ij}$.
	Because the $\xi_i$ are $\mathcal{X}^i$-invariant, we have $\xi_i=\xi_i(\varphi^1,\varphi^2)$.
	Consequently the distributions $\{\nabla\varphi^1,\nabla\varphi^2\}$ and $\{\nabla\xi_1,\nabla\xi_2\}$ coincide, and both are integrable as $[\nabla\varphi^i,\nabla\varphi^j]=0$.
	
	Because the reduction $\Phi(\varphi^1,\theta_1,\varphi^2,\theta_2)=(\varphi^1,\varphi^2)$ is $\mathcal{X}^1$-$\mathcal{X}^2$ invariant, this is generically a Riemannian submersion, and we have a metric and complex structure at interior points of $\Sigma^2$.
	Abbreviating $\mathcal{V}=det(G^{ij})=|\nabla\varphi^1|^2|\nabla\varphi^2|^2-\big<\nabla\varphi^1,\nabla\varphi^2\big>{}^2$,
	\begin{equation}
		g_\Sigma=(G_{ij}) \quad\text{and}\quad
		J_\Sigma
		\;=\;
		\mathcal{V}^{-\frac12}\left(\begin{array}{cc}
			\left<\mathcal{X}{}^1,\,\mathcal{X}{}^2\right> & -|\mathcal{X}{}^1|^2 \\
			|\mathcal{X}{}^2|^2 & -\left<\mathcal{X}{}^1,\,\mathcal{X}{}^2\right>
		\end{array}\right).
	\end{equation}
	The metric $g_\Sigma$ is inherited from the parent manifold $(M^4,g)$, but the complex structure $J_\Sigma$ is not.
	Rather, it is the (dualization of) the Hodge-* on $(\Sigma^2,g_\Sigma)$, and is obtained from $J_{\Sigma}{}^i_j=(det\,g_{ij})g_{ik}\epsilon^{jk}$ where $\epsilon$ is the Levi-Civita symbol.
	Citing Proposition 2.5 of \cite{Web1}, we have degenerate-elliptic equations for $\varphi^1$, $\varphi^2$
	\begin{equation}
		d(\mathcal{V}^{-\frac12}J_\Sigma{d}\varphi^k)=0. \label{EqnEllipticVar}
	\end{equation}
	From Proposition 2.2 of \cite{Web1}, the Ricci and scalar curvatures on $M^4$ are
	\begin{eqnarray}
		\begin{aligned}
			\rho\;=\;-\sqrt{-1}\partial\bar\partial\log\,\mathcal{V},
			\quad\quad
			s\;=\;-\triangle\log\,\mathcal{V}.
		\end{aligned} \label{EqnsRicScal}
	\end{eqnarray}
	%From (\ref{EqnLaplRelation}) and (\ref{EqnsRicScal}), the scalar curvature $s$ on $M^4$ can be computed on $\Sigma^2$, where
	The scalar curvature $s$ on $M^4$ can be computed on $\Sigma^2$, where
	\begin{equation}
		\triangle_{\Sigma}\mathcal{V}^{\frac12}+\frac12s\mathcal{V}^{\frac12}\;=\;0. \label{EqnScalarCurvatureEquation}
	\end{equation}
	This is equivalent to the 4th-order Abreu equation for scalar curvature $s=-\frac12\frac{\partial^2G^{ij}}{\partial\varphi^i\varphi^j}$ (see Eq.~(10) of \cite{Ab1}), after a Trudinger-Wang style reduction \cite{TW02}.
	See for example Corollary 2.4 of \cite{Web1}.
	Therefore when the $M^4$ scalar curvature is zero there exists a natural analytic coordinate $z=x+\sqrt{-1}y$ on $\Sigma^2$ given by $y=\sqrt{\mathcal{V}}$, and $x$ is an harmonic conjugate of $y$.
	Because $y=\sqrt{\mathcal{V}}$ is a volume, we call the isothermal system $(x,y)$ the \textit{volumetric normal coordinates} and the analytic coordinate $z=x+\sqrt{-1}y$ the \textit{volumetric normal function}.
	From Proposition 1.1 of \cite{Web1}, when $\Sigma^2$ is closed and has connected boundary, the map to the upper half-space
	\begin{equation}
		z:\Sigma^2\rightarrow\{y\ge0\}
	\end{equation}
	is a biholomorphism with $\partial\Sigma^2$ mapping bijectively onto the half-space boundary $\{y=0\}$.
	In \cite{Web3}, the Ricci form $\rho$ and the Weyl tensor on $M^4$ were both computed from data on its reduction $\Sigma^2$.
	For example letting $K_\Sigma$ be the Gaussian curvature on $(\Sigma^2,g_\Sigma)$, Proposition A.2 of \cite{Web3} states $|W^-|^2=24(K_\Sigma)^2$.

	%
	%
	%
	% Section 2.2
	%
	%
	%
	\subsection{The K\"ahler and symplectic potential} \label{SubSecPotentials}
	We describe the so-called K\"ahler and symplectic potentials on $M^4$.
	Let $P^2\subset{}M^4$ be the union of the polar submanifolds, so in particular $\Phi:P^2\rightarrow\Sigma^2$ maps $P$ onto $(\partial\Sigma^2)_I$ and $\Phi:M^4\setminus{}P^2\rightarrow\Sigma^2$ is a Riemannian submersion onto the interior of $\Sigma^2$.
	Any leaf of the $\{\nabla\varphi^1,\nabla\varphi^2\}$-distribution maps bijectively onto $\Sigma^2$, and $M^4\setminus{}P^2$ is diffeomorphic to the product of $Int(\Sigma^2)$ with any of the $\mathcal{X}^1$-$\mathcal{X}^2$ leafs.
	Because $Int(\Sigma^2)$ is contractable, we can write $\omega=\sqrt{-1}\partial\bar\partial{F}$ for some smooth pseudoconvex function ${F}={F}(\xi^1,\xi^2)$ called the \textit{K\"ahler potential} of $g$.
	Because ${F}$ is pseudoconvex and the distribution $\{\nabla\xi^1,\nabla\xi^2\}$ is Legendrian---meaning $\omega(\nabla\xi^i,\nabla\xi^j)=0$ for $i,j=1,2$---we have that ${F}$ is a \emph{convex} function of $\xi^1,\xi^2$.
	From $d\varphi^i=\omega\big(\frac{\partial}{\partial\theta_i},\,\big)$ and $d\xi_i=Jd\theta_i$,
	\begin{equation}
		\varphi^1\;=\;\frac{\partial{F}}{\partial\xi_1}, \quad
		\varphi^2\;=\;\frac{\partial{F}}{\partial\xi_2},
	\end{equation}
	see for example Eq. (4.4) of \cite{G}.
	These are precisely the hodographs usually seen with the Legendre transform; indeed because $F=F(\xi_1,\xi_2)$ is convex it has a Legendre transform $G=G(\varphi^1,\varphi^2)$, called the manifold's \textit{symplectic potential}.
	Referring to the $2\times2$ matrix $G_{ij}$ above, it turns out that $G_{ij}=\frac{\partial^2{G}}{\partial\varphi^i\partial\varphi^j}$.
	See \cite{G} for a justification of this fact.
	We remark that $G(\varphi^1,\varphi^2)$ is a smooth convex function on $M^4\setminus{}P^2$, and passes to a smooth convex function on $Int(\Sigma^2)$.

	%
	%
	%
	% Section 2.3: Conformal Change
	%
	%
	%
	\subsection{A conformal change of the metric} \label{SubSecConfChangeOfMetric}
	
	We compute the Gaussian curvature $K_\Sigma$ of $\Sigma^2$ and create a conformal-change formula (\ref{EqnConformalScalar}) that will be used in Section \ref{SectionNoEdges}.
	We use ``$s_\Sigma$'' to indicate the scalar curvature of $(\Sigma^2,g_\Sigma)$, and ``$s$'' to indicate the scalar curvature on the parent 4-manifold $(M^4,g)$, which is $\mathcal{X}^1$-$\mathcal{X}^2$ invariant so passes to a function on $\Sigma^2$.
	Of course $K_\Sigma=\frac12s_\Sigma$.
	Due to the fact that $G_{ij}$ is a matrix of second partial derivatives, the metric on $g_\Sigma=(G_{ij})$ on $\Sigma^2$ obeys
	\begin{eqnarray}
		\frac{\partial{G}_{ij}}{\partial\varphi^k}\;=\;\frac{\partial{G}_{ik}}{\partial\varphi^j}. \label{EqnPseudoKahler}
	\end{eqnarray}
	Using $\mathcal{V}=\text{det}(G^{ij})=|\nabla\varphi^1|^2|\nabla\varphi^2|^2-\big<\nabla\varphi^1,\nabla\varphi^2\big>{}^2$, the Christoffel symbols are
	\begin{eqnarray}
		\begin{aligned}
			\Gamma_{ij}^k\;=\;\frac12\frac{\partial{G}_{ij}}{\partial\varphi^s}G^{sk}, \quad\quad
			\Gamma^k
			\;\stackrel{\textit{def}}{=}\;G^{ij}\Gamma_{ij}^k
			\;=\;-G^{ks}\frac{\partial}{\partial\varphi^s}\log\mathcal{V}^{\frac12}.
		\end{aligned} \label{EqnChristoSymbs}
	\end{eqnarray}
	The usual formula for scalar curvature in terms of Christoffel symbols is
	\begin{eqnarray}
		s_{\Sigma}\;=\;
		G^{ij}\frac{\partial\Gamma_{ij}^s}{\partial\varphi^s}
		-G^{ij}\frac{\partial\Gamma_{sj}^s}{\partial\varphi^i}
		+G_{st}\Gamma^s\Gamma^t
		\,-\,G^{is}G^{jt}G_{kl}\Gamma_{ij}^k\Gamma_{st}^l.
	\end{eqnarray}
	From (\ref{EqnPseudoKahler}) and (\ref{EqnChristoSymbs}) we find that $\frac{\partial}{\partial\varphi^s}\left(G^{ij}\Gamma_{ij}^s\right)-\frac{\partial}{\partial\varphi^i}\left(G^{ij}\Gamma_{sj}^s\right)=0$, so
	\begin{eqnarray}
		\begin{aligned}
			s_{\Sigma}
			&\;=\;
			\frac{\partial}{\partial\varphi^s}\left(G^{ij}\Gamma_{ij}^s\right)-\frac{\partial}{\partial\varphi^i}\left(G^{ij}\Gamma_{sj}^s\right)
			+G^{is}G^{jt}G_{kl}\Gamma_{ij}^k\Gamma_{st}^l-G_{st}\Gamma^s\Gamma^t \\
			&\;=\;
			G^{is}G^{jt}G_{kl}\Gamma_{ij}^k\Gamma_{st}^l-G_{st}\Gamma^s\Gamma^t \\
			&\;=\;
			\left|\Gamma_{ij}^k\right|^2\,-\,\left|\Gamma^k\right|^2
			\;=\;
			\left|\Gamma_{ij}^k\right|^2\,-\,\big|\nabla\log\mathcal{V}^{\frac12}\big|^2.
		\end{aligned} \label{EqnCoordSigmaScalarComp}
	\end{eqnarray}
	Now modify the metric by $\widetilde{g}_{\Sigma}=\mathcal{V}^{\frac12}{g}_\Sigma$.
	The usual conformal-change formula gives
	\begin{eqnarray}
		\begin{aligned}
			\widetilde{s}_\Sigma
			&\;=\;\mathcal{V}^{-\frac12}\left(s_\Sigma\,-\,\triangle_\Sigma\log\mathcal{V}^{\frac12}\right) \\
			&\;=\;\mathcal{V}^{-\frac12}\left(
			\left|\Gamma_{ij}^k\right|^2\,-\,\big|\nabla\log\mathcal{V}^{\frac12}\big|^2
			\,-\,\triangle_\Sigma\log\mathcal{V}^{\frac12}\right).
		\end{aligned} \label{EqnConfChangeScalar}
	\end{eqnarray}
	But $\triangle_\Sigma\log\mathcal{V}^{\frac12}+\big|\nabla\log\mathcal{V}^{\frac12}\big|{}^2=\mathcal{V}^{-\frac12}\triangle_\Sigma\mathcal{V}^{\frac12}$, so along with (\ref{EqnsRicScal}) we see the conformally related scalar curvature is simply
	\begin{eqnarray}
		\begin{array}{ll}
			\widetilde{s}_\Sigma
			&\;=\;\mathcal{V}^{-\frac12}\left(\left|\Gamma_{ij}^k\right|^2\,-\,\mathcal{V}^{-\frac12}\triangle_\Sigma\mathcal{V}^{\frac12}\right)
			\;=\;\mathcal{V}^{-\frac12}\left(\left|\Gamma_{ij}^k\right|^2\,+\,\frac12s\right).
		\end{array} \label{EqnConformalScalar}
	\end{eqnarray}
	In particular $s\ge0$ on $M^4$ implies $\widetilde{s}_\Sigma\ge0$.
	This is the central observation behind the proof of Theorem \ref{ThmComplete}.

	%
	%
	%
	%
	%
	% Sec 3: 
	%
	%
	%
	%
	%
	\section{Characteristics of the reduction}
	
	We characterize points of $\partial\Sigma^2$ and prove Propositions \ref{PropConvexity}, \ref{PropBdStruct1}, \ref{PropBdStruct2}, and \ref{PropBdStruct3}.
	
	%
	%
	%
	% SubSec 3.1: included and non-included boundary points
	%
	%
	%
	
	\subsection{Structure included and non-included boundary points}
	In the $\varphi^1$-$\varphi^2$ plane it is convenient to refer to a coordinate disk around a point $p=(p^1,p^2)$ by
	\begin{equation}
		D_p(r)\;=\;\left\{\,(\varphi^1,\varphi^2)\in\mathbb{R}^2 \;\big|\; \sqrt{(\varphi^1-p^1)^2+(\varphi^2-p^2)^2}<r\,\right\}.
	\end{equation}
	We remark that $\Phi^{-1}(D_p(r))\subseteq{}M^4$ is an equivariant neighborhood of any $\tilde{p}\in\Phi^{-1}(p)\subset{}M^4$.
	Proposition \ref{PropStructIncluded} is well-known in the case of compact symplectic manifolds \cite{Delzant}, \cite{SymplMS}; we give a geometrically-flavored proof in the K\"ahler setting.
	\begin{proposition}[Structure of $\Sigma^2$ near included boundary points] \label{PropStructIncluded}
		Assume $p=(p^1,p^2)$ is a point of $(\partial\Sigma^2)_I$.
		Then on $\Phi^{-1}(p)$ the distribution $\{\mathcal{X}^1,\mathcal{X}^2\}$ has rank $0$ or $1$, and a coordinate disk $D_p(\epsilon)$ exists so that one of the following holds:
		\begin{enumerate}
			\item the distribution $\{\mathcal{X}^1,\mathcal{X}^2\}$ has rank 1, and a linear function $m(\varphi^1,\varphi^2)=\alpha(\varphi^1-p^1)+\beta(\varphi^2-p^2)$ exists so that $D_p(\epsilon)\cap\Sigma^2=D_p(\epsilon)\cap\{m\ge0\}$.
			\item the distribution $\{\mathcal{X}^1,\mathcal{X}^2\}$ has rank 0, then two linear functions $m_0(\varphi^1,\varphi^2)=\alpha_0(\varphi^1-p^1)+\beta_0(\varphi^2-p^2)$,
			$m_1(\varphi^1,\varphi^2)=\alpha_1(\varphi^1-p^1)+\beta_1(\varphi^2-p^2)$ exist so that $D_p(\epsilon)\cap\Sigma^2=D_p(\epsilon)\cap\{m_0\ge0\}\cap\{m_1\ge0\}$.
		\end{enumerate}
		In particular, $\Sigma^2\cap{}D_p(\epsilon)$ is convex.
	\end{proposition}
	\begin{proof}
		Let $p\in{}M^4$ be any point in $M^4$ where $\{\mathcal{X}^1,\mathcal{X}^2\}$ has rank 2.
		Then the distribution $\{\nabla\varphi^1,\nabla\varphi^2\}$ also has rank 2, so therefore near $p$ the projection $\Phi:M^4\rightarrow\Sigma^2$ is a Riemannian submersion, and maps neighborhoods of $p$ to neighborhoods $\Phi(p)$.
		Thus if the rank of the distribution $\{\nabla\varphi^1,\nabla\varphi^2\}$ is $2$, then $\Phi(p)\notin{}(\partial\Sigma^2)_I$.
		Therefore $\Phi(p)\in{}(\partial\Sigma^2)_I$ implies the rank is 0 or 1.
		
		Now assume the rank of $\{\mathcal{X}^1,\mathcal{X}^2\}$ is 1; then numbers $\alpha,\beta$ exist so the Killing field $\mathcal{X}=\alpha\mathcal{X}^1+\beta\mathcal{X}^2$ vanishes at $p$.
		The zero-set $\{\mathcal{X}=0\}$ of any Killing field is a totally geodesic submanifold of even dimension.
		Because the rank of the distribution at $p$ is $1$, and because the Killing fields commute, a second non-zero Killing field exists that preserves the zero-set of $\mathcal{X}$, so the zero-set $\{\mathcal{X}=0\}$ through $p$ has dimension at least $1$.
		Recalling that the dimension of any zero-locus of any Killing field on any manifold has even dimension, the zero-set $L^2=\{\mathcal{X}=0\}$ is 2-dimensional.
		Because $\mathcal{X}$ is a holomorphic vector field also $L^2$ is a complex submanifold.
		
		Setting $\varphi=\alpha\varphi^1+\beta\varphi^2$, we have that $\nabla\varphi=-J\mathcal{X}=0$ on $L^2$; thus $L^2$ consists of critical points of the function $\varphi$.
		To see that $L^2$ is either a local max or a min of $\varphi$ (not, for example, a saddle), consider any short segment perpendicular to $L^2$, and consider its orbit via the field $\mathcal{X}$.
		Since $\mathcal{X}$ is zero on $L^2$, one endpoint of the segment will remain fixed on $L^2$ while the other endpoint creates a circle; the orbit of the segment therefore forms a disk.
		Because $\varphi$ is also $\mathcal{X}$-invariant, when restricted to this disk the level-sets of $\varphi$ are the integral curves of $\mathcal{X}$, which are therefore circles.
		Because its level-sets are compact, $\varphi$ either has a max or a min at the disk's center.
		Along $L^2$ the function $\varphi$ remains constant, so therefore the entire locus $L^2$ is a local max or min of $\varphi$.
		Replacing our choice of $\alpha$, $\beta$ with $-\alpha$, $-\beta$ if necessary, we may assume $L^2$ is a local minimum of $\varphi$.
		Finally, from $\varphi^1(p)=p^1$ and $\varphi^2(p)=p^2$, we have $\alpha(\varphi^1-p^1)+\beta(\varphi^2-p^2)\ge0$ in a neighborhood of $L^2$ with equality if and only if $(\varphi^1,\varphi^2)\in{}L^2$.
		
		In the case that $\{\mathcal{X}^1,\mathcal{X}^2\}$ has rank 0, then it is the intersection of two totally-geodesic submanifolds on which the distribution has rank $1$.
		Following the above argument twice, once for each submanifold, leads to two linear relations.
	\end{proof}
	In the case of complete K\"ahler manifolds, the reduction $\Sigma^2$ might not contain all of its boundary points.
	We show that when the parent $M^4$ is complete, these non-included points are ``points at infinity'' in a precise sense.
	\begin{proposition}[Structure of $\Sigma^2$ near non-included boundary points] \label{PropNonInclBPt}
		Assume $(M^4,J,g,\mathcal{X}^1,\mathcal{X}^2)$ is geodesically complete, and let $q\in(\partial\Sigma^2)_N$ be a non-included boundary point.
		Then if $\{p_i\}\subset{}M^4$ is any sequence of points so that $\Phi(p_i)\rightarrow{}q$ in the coordinate topology on $\mathbb{R}^2$, then $p_i$ diverges in the metric topology on $M^4$.
	\end{proposition}
	\begin{proof}
		First, if $K\subset{}M^4$ is any compact set, then its image $\Phi(K)$ is compact in the coordinate topology on $\mathbb{R}^2$ by continuity of $\Phi$.
		Given such a compact set $K\subset{}M^4$ we show that at most finitely many $p_i$ lie in $K$.
		If not, there is a subsequence (still denoted $p_i$) that remains within $K$ but so that $\Phi(p_i)$ converges to $p$ in the $\mathbb{R}^2$-topology.
		But then $\Phi(p_i)$ remains within the compact set $\Phi(K)\in\Sigma^2$ so a subsequence converges to some point in $\Phi(K)$.
		But this is impossible, because $\Phi(p_i)$ converges to $q\notin\Sigma^2$.
	\end{proof}

	%
	%
	%
	% Sec 3.2: convexity
	%
	%
	%
	\subsection{Convexity of the reduction}
	In the purely symplectic setting, if $M^4$ is non-compact it is impossible to recover the classic result of \cite{Delzant} that reductions are polygons when $M^4$ is compact.
	But if $M^4$ has a complete K\"ahler metric we recover convexity.
	The argument is based on the following very basic computation: letting $\textbf{a}=(a^1,a^2)$ and $\textbf{b}=(b^1,b^2)$ be points in the $\varphi^1$-$\varphi^2$ plane and $\gamma(t)=(1-t)\textbf{a}+t\textbf{b}$ the segment between them, then $\frac{d}{dt}=
	(b^1-a^1)\frac{\partial}{\partial\varphi^1}
	+(b^2-a^2)\frac{\partial}{\partial\varphi^2}$ and because the coefficients are constant we have
	\begin{equation}
		\frac{d^2}{dt^2}\;=\;
		(b^1-a^1)^2\frac{\partial^2}{(\partial\varphi^1)^2}
		+2(b^1-a^1)(b^2-a^2)\frac{\partial^2}{\partial\varphi^1\partial\varphi^2}
		+(b^2-a^2)^2\frac{\partial^2}{(\partial\varphi^2)^2}.
	\end{equation}
	Because $g_{ij}=\frac{\partial^2G}{\partial\varphi^i\varphi^j}$, along this segment we have
	\begin{equation}
		\small
		\begin{aligned}
			\frac{d^2G}{dt^2}
			&=
			(b^1-a^1)^2\frac{\partial^2G}{(\partial\varphi^1)^2}
			+2(b^1-a^1)(b^2-a^2)\frac{\partial^2G}{\partial\varphi^1\partial\varphi^2}
			+(b^2-a^2)^2\frac{\partial^2G}{(\partial\varphi^2)^2} \\
			&=\;g\left(\frac{d}{dt},\,\frac{d}{dt}\right)\;>\;0. \label{EqnSegmentLength}
		\end{aligned}
	\end{equation}
	\begin{proposition}[\textit{cf.} Proposition \ref{PropConvexity}] \label{PropConvex}
		Assume the metric on $(M^4,J,g,\mathcal{X}^1,\mathcal{X}^2)$ is complete.
		Then its reduction $\Sigma^2$ is convex.
	\end{proposition}
	\begin{proof}
		Let ${\boldsymbol\varphi}_0,{\boldsymbol\varphi}_1\in\mathbb{R}^2$ be points within $Int(\Sigma^2)$ under the condition that the line segment $\gamma(t)=(1-t){\boldsymbol\varphi}_0+t{\boldsymbol\varphi}_1$, $t\in[0,1]$ between them remains in $\Sigma^2$.
		From (\ref{EqnSegmentLength}),
		\begin{equation}
			g\left(\frac{d}{dt},\,\frac{d}{dt}\right)
			\;=\;\frac{d^2G}{dt^2}, \quad\text{and}\quad
			\left|\frac{d}{dt}\right|
			\;=\;\sqrt{\frac{d^2G}{dt^2}}. \label{EqnGConcave}
		\end{equation}
		If the segment lies within $Int(\Sigma^2)$, its length is bounded: from H\"older's inequality
		\begin{equation}
			\begin{aligned}
				Len(\gamma)
				&\;=\;\int_0^1\left|\frac{d}{dt}\right|\,dt
				\;=\;\int_0^1\left(\frac{d^2G}{dt^2}\right)^{\frac12}\,dt \\
				&\;\le\;\left(\int_0^1\frac{d^2G}{dt^2}\,dt\right)^{\frac12}
				\;=\;\left(\frac{dG}{dt}({\boldsymbol\varphi}_1)\;-\;\frac{dG}{dt}({\boldsymbol\varphi}_0)\right)^{\frac12}
			\end{aligned} \label{EqnLenEst}
		\end{equation}
		so the length of the line segment is bounded in terms of its endpoints.
		Because $Int(\Sigma^2)$ is connected we can vary ${\boldsymbol\varphi}_0$, ${\boldsymbol\varphi}_1$ throughout $Int(\Sigma^2)$.
		If $\Sigma^2$ is non-convex then there exists some point $p\in(\partial\Sigma^2)_N$ that has a supporting segment that is in the closure of $\Sigma^2$ and has endpoints in $\Sigma^2$.
		By Proposition \ref{PropNonInclBPt} the segment $\gamma$ must have infinite length.
		However (\ref{EqnLenEst}) gives a uniform upper bound on the length of $l$.
		From this contradiction we conclude that $\Sigma^2$ is convex.
	\end{proof}

	%
	%
	%
	% Sec 3.3: Structure of interior points
	%
	%
	%
	\subsection{Structure of interior points}
	
	\begin{lemma} \label{LemmaStructInteriorPts}
		If $p\in\Sigma^2$ then its preimage $\Phi^{-1}(p)\subset{}M^4$ is a connected submanifold with tangent bundle $\text{span}\{\mathcal{X}^1,\mathcal{X}^2\}\subset{}TM^4$ restricted to $\Phi^{-1}(p)$.
		Finally, $p\in{}Int(\Sigma^2)$ if and only if $\{\mathcal{X}^1,\mathcal{X}^2\}$ has rank $2$ on $\Phi^{-1}(p)$, if and only if a coordinate disk $D_p(\epsilon)$ exists so $\Phi:\Phi^{-1}(D_p(\epsilon))\rightarrow{}D_p(\epsilon)$ is a Riemannian submersion.
	\end{lemma}
	\begin{proof}
		Because $\mathcal{X}^1$, $\mathcal{X}^2$ are Killing, the orbit of any point $\tilde{p}\in{}M^4$ under $\mathcal{X}^1$, $\mathcal{X}^2$ is a manifold with tangent bundle being the distribution $\{\mathcal{X}^1,\mathcal{X}^2\}$ restricted to the orbit.
		Because $\Sigma^2$ is $M^4$ modulo the action of these fields, the image under $\Phi$ of the orbit of any point $\tilde{p}\in{}M^4$ is a single point in $\Sigma^2$.
		The distribution $span\{\nabla\varphi^1,\nabla\varphi^2\}$ has the same rank as $span\{\mathcal{X}^2,\mathcal{X}^2\}$ and is perpendicular to it, so the pushforward $\Phi_*$ is an isometry on $span\{\nabla\varphi^1,\nabla\varphi^2\}$.
		Therefore, because $span\{\nabla\varphi^1,\nabla\varphi^2\}$ has rank 2, $\Phi$ is a Riemannian submersion on a neighborhood of $p$.
		
		It remains to show that a point of $\Sigma^2$ has connected pre-image.
		Assume a point $p\in\Sigma^2$ has two components $T_1,T_2\subset\Phi^{-1}(\Sigma^2)$.
		Let $\gamma(s)$, $s\in[0,1]$ be a geodesic from $T_1$ to $T_2$ that represents a shortest path between these submanifolds.
		The vector $\dot\gamma$ is pependicular to $T_1$ and $T_2$ at the two points of contact: $\left<\dot\gamma(0),\mathcal{X}^i\right>=\left<\dot\gamma(1),\mathcal{X}^i\right>=0$ for both fields $\mathcal{X}^1$, $\mathcal{X}^2$.
		To show $\left<\dot\gamma(t),\mathcal{X}^i\right>=0$ for all $t$, we compute
		\begin{equation}
			\begin{aligned}
				\frac{d^2}{ds^2}\left<\mathcal{X},\dot\gamma\right>
				&=\dot\gamma\dot\gamma\left<\mathcal{X},\,\dot\gamma\right>
				=\left<\nabla_{\dot\gamma}\nabla_{\dot\gamma}\mathcal{X},\,\dot\gamma\right>
				=\left<\Riem(\dot\gamma,\mathcal{X})\dot\gamma,\,\dot\gamma\right>
				=0.
			\end{aligned}
		\end{equation}
		In particular $\dot\gamma\in{}span\{\nabla\varphi^1,\nabla\varphi^2\}$ so the path $s\mapsto\Psi(\gamma(s))$ in $\Sigma^2$ has the same length as the path $s\mapsto\gamma$ in $M^4$.
		Given any $s\in[0,1]$, draw the line segment in $\Sigma^2$
		\begin{equation}
			\eta_s(t)=(1-t)\,\Phi(\gamma(0))\,+\,t\,\Phi(\gamma(s)) \label{EqnSpanningSurface}
		\end{equation}
		from $\Phi(\gamma(0))$ to $\Phi(\gamma(s))$.
		Then $(s,t)\mapsto\eta_s(t)$ in $\Sigma^2$ is a spanning surface for the closed loop $\Phi(\gamma(s))$.
		
		We can lift each segment $t\mapsto\eta_s(t)$ to a path in $M^4$ in the following way.
		In action-angle coordinates, $\gamma(s)$ has the expression $\gamma(s)=(\varphi^1(s),\varphi^2(s),c_1,c_2)$ due to the fact that, as we have shown, the angle coordinates $(\theta_1,\theta_2)=(c_1,c_2)$ remain constant along $\gamma$.
		Then for each $s$ the path (\ref{EqnSpanningSurface}) lifts directly to $M^4$ by using the $(\varphi^1,\varphi^2)$-coordinates from its expression in $\Sigma^2$ and assigning it the $(\theta_1,\theta_2)$ coordinates $(c_1,c_2)$.
		This produces a continuous lift of $(s,t)\rightarrow\eta_s(t)$ in $\Sigma^2$ to a 2-disk $(s,t)\mapsto\tilde{\eta}_s(t)$ in $M^4$.
		
		Because the path $t\mapsto\eta_t(1)$ is a trivial path in $\Sigma^2$, therefore its lift $t\mapsto\tilde{\eta}_t(1)$ remains on a single leaf of the $\{\mathcal{X}{}^1,\mathcal{X}{}^2\}$ distribution.
		This contradicts the assumption that $\gamma(0)$ and $\gamma(1)$ are on different leaves of this distribution.
	\end{proof}

	%
	%
	%
	%
	%
	% Section 4: Open reductions
	%
	%
	%
	%
	%
	
	\section{Reductions without an included boundary} \label{SectionNoEdges}
	
	We prove Theorem \ref{ThmComplete}, that if $s\ge0$ and $(\partial\Sigma^2)_I=\varnothing$ then $M^4$ is flat.
	We do so in the following stages.
	Because $\Phi:M^4\rightarrow\Sigma^2$ is now a global Riemannian submersion by Lemma \ref{LemmaStructInteriorPts}, in particular $(\Sigma^2,g_\Sigma)$ is complete.
	Recalling the quantity $\mathcal{V}$ from Section \ref{SubSecRedProps}, we first show that if $\mathcal{V}$ is constant, the metric is flat.
	By (\ref{EqnScalarCurvatureEquation}) the scalar curvature condition $s\ge0$ is equivalent to $\triangle_\Sigma\sqrt{\mathcal{V}}\le0$; we use a classical Liouville theorem to show this implies $\mathcal{V}$ is constant.
	
	But to use the classical Liouville theorem, me must determine that the complete manifold $(\Sigma^2,g_\Sigma)$ is actually biholomorphic to $\mathbb{C}$, or, equivalently, that it is simply connected and parabolic in the sense of potential theory \cite{CY}.
	To do this, we show that if $(\Sigma^2,g_\Sigma)$ is complete then the conformal change from \S\ref{SubSecConfChangeOfMetric} gives a metric $\widetilde{g}_\Sigma$ that actually remains complete, and, crucially, gives a manifold $(\Sigma^2,\widetilde{g}_\Sigma)$ with non-negative Gaussian curvature.
	Thus $\Sigma^2$ is biholomorphic with $\mathbb{C}$, by the Cheng-Yau criterion for parabolicity.
	The classical Liouville theorem then forces $\mathcal{V}$ to be constant, and we conclude that $g_\Sigma$ is a flat metric.
	\begin{lemma} \label{LemmaTConstFlat}
		Assume $(\Sigma^2,g_\Sigma)$ is geodesically complete.
		If $\mathcal{V}$ is constant, then $(\Sigma^2,g_\Sigma)$ is biholomorphic to $\mathbb{C}$, and $g_\Sigma$ is a flat metric with constant coefficients when expressed in $\varphi^1$, $\varphi^2$ coordinates.
	\end{lemma}
	\begin{proof}
		With $\mathcal{V}$ constant, then $d(\mathcal{V}^{-\frac12}J_\Sigma{d}\varphi^k)=0$ from (\ref{EqnEllipticVar}) gives $d(J_\Sigma{}d\varphi^k)=0$, so $\varphi^1$, $\varphi^2$ are harmonic on $\Sigma^2$.
		Therefore each determines an analytic function which we can write $z=\varphi^1+\sqrt{-1}\eta^1$, $w=\varphi^2+\sqrt{-1}\eta^2$ (where $d\eta^k=J_\Sigma{}d\varphi^k$).
		Away from possible critical points these are each a holomorphic coordinate on $\Sigma^2$, and 
		\begin{eqnarray}
			\begin{aligned}
				\frac{d}{dz}&\;=\;\frac{1}{2|\nabla\varphi^1|^2}\left(\nabla\varphi^1\,+\,\sqrt{-1}J_\Sigma\nabla\varphi^1\right) \\
				\frac{d}{dw}&\;=\;\frac{1}{2|\nabla\varphi^2|^2}\left(\nabla\varphi^2\,+\,\sqrt{-1}J_\Sigma\nabla\varphi^2\right)
			\end{aligned}
		\end{eqnarray}
		and we easily compute the transition function:
		\begin{eqnarray}
			\frac{dw}{dz}
			\;=\;g\left(\frac{d}{dz},\,\overline{\nabla}{w}\right)
			\;=\;\frac{\left<\nabla\varphi^1,\,\nabla\varphi^2\right>}{|\nabla\varphi^1|^2}\,-\,\sqrt{-1}\frac{\sqrt\mathcal{V}}{|\nabla\varphi^1|^2}.
		\end{eqnarray}
		But $\frac{dw}{dz}$ is holomorphic, so its imaginary part is harmonic, so $|\nabla\varphi^1|^{-2}$ is harmonic.
		In the $z$-coordinate the Hermitian metric is $h_\Sigma=\left|\frac{d}{dz}\right|^2=\frac12|\nabla\varphi^1|^{-2}$, so $h_\Sigma$ is an harmonic function.
		Using $\triangle_\Sigma{}h_\Sigma=0$, the Gaussian curvature is
		\begin{eqnarray}
			K_\Sigma
			\;=\;-\frac12h_\Sigma^{-1}\triangle_\Sigma\log{h}_\Sigma\;=\;4|\nabla|\nabla\varphi^1||^2 \label{EqnSurfaceK}
		\end{eqnarray}
		which is non-negative, forcing the complete manifold $(\Sigma^2,g_\Sigma)$ to be parabolic---this is due to the Cheng-Yau condition for parabolicity; see \cite{CY} or the remark below.
		The region $\Sigma^2\subset\mathbb{R}^2$ is simply connected due to the fact that it is convex by Proposition \ref{PropConvex}.
		Consequently $(\Sigma^2,J_{\Sigma})$ is biholomorphic to $\mathbb{C}$.
		But then the classic Liouville theorem implies $|\nabla\varphi^1|^{-2}$ is actually constant, because it is an harmonic function on $\mathbb{C}$ bounded from below.
		In particular, (\ref{EqnSurfaceK}) now gives $K_\Sigma=0$.
		
		Similarly $|\nabla\varphi^2|^{-2}$ is constant.
		The fact that $\mathcal{V}$ is constant and $\mathcal{V}=|\nabla\varphi^1|^2|\nabla\varphi^2|^2-\left<\nabla\varphi^1,\nabla\varphi^2\right>{}^2$ means $\left<\nabla\varphi^1,\nabla\varphi^2\right>$ is constant.
		Because $g_\Sigma{}^{ij}=\left<\nabla\varphi^i,\nabla\varphi^j\right>$, all components of the metric are constants when measured in $(\varphi^1,\varphi^2)$ coordinates.
	\end{proof}
	%Lemma \ref{LemmaTConstFlat} gives flatness of $g_\Sigma$ provided we somehow know $\mathcal{V}$ is a constant.
	%The conditions under which we know this is clarified below in Lemma \ref{ThmRTwoClassification}.
	%The next lemma gives a hint: if the reduction is biholomorphic to $\mathbb{C}$, then $\mathcal{V}$ is constant.
	\begin{lemma} \label{LemmaBiholoToCFlat}
		Assume $(\Sigma^2,g_\Sigma)$ is geodesically complete, $(\Sigma^2,J_\Sigma)$ is biholomorphic to $\mathbb{C}$, and $\triangle_\Sigma\sqrt{\mathcal{V}}\le0$.
		Then $(\Sigma^2,g_\Sigma)$ is flat and the metric $g_\Sigma$ has constant coefficients when expressed in $\varphi^1$, $\varphi^2$ coordinates.
	\end{lemma}
	\begin{proof}
		Since the function $\mathcal{V}^{\frac12}$ is superharmonic and bounded from below on $\mathbb{C}$, it is constant.
		Lemma \ref{LemmaTConstFlat} now provides the result.
	\end{proof}
	Lemmas \ref{LemmaTConstFlat} and \ref{LemmaBiholoToCFlat} together show $\mathcal{V}=const$ if and only if $(\Sigma^2,J_\Sigma)$ is biholomorphic to $\mathbb{C}$.
	%In Theorem \ref{ThmRTwoClassification} we prove biholomorphicity to $\mathbb{C}$, assuming only geodesic completeness and $\triangle_\Sigma\sqrt{\mathcal{V}}\le0$.
	Before this equivalence can be made useful, we require a fact about the conformal change discussed in \S\ref{SubSecConfChangeOfMetric}.
	\begin{lemma} \label{LemmaConformalComplete}
		Assume $(\Sigma^2,g_\Sigma)$ is geodesically complete and $\triangle_\Sigma\sqrt{\mathcal{V}}\le0$.
		Setting $\widetilde{g}_\Sigma=\mathcal{V}^{\frac12}g_\Sigma$, then $(\Sigma^2,\widetilde{g}_\Sigma)$ is also geodesically complete.
	\end{lemma}
	\begin{proof}
		For a proof by contradiction suppose $(\Sigma^2,\tilde{g}_\Sigma)$ is not geodesically complete.
		Let $\gamma:[0,R]\rightarrow\Sigma^2$ be a shortest inextendible unit-speed geodesic in the $\widetilde{g}_\Sigma$ metric.
		Then $\gamma$ must have \emph{infinite} length in the $g_\Sigma$ metric; the reason is that, because the conformal factor $\sqrt{\mathcal{V}}$ is smooth and non-negative, if $\gamma$ has finite total length in $g_\Sigma$ it will remain within a compact subset $K$ of $\Sigma^2$.
		But on any compact subset the conformal factor $\mathcal{V}$ is smooth and nonzero, so the defining ODE for $\gamma$ remains smooth in $\tilde{g}_{\Sigma}$, so it is extendible.
		
		For a contradiction, we prove that $\gamma$ has finite length in the $g_\Sigma$ metric.
		Set $p=\gamma(0)$ and let $r$ be the distance function $r=dist(p,\cdot)$ for the $\tilde{g}_\Sigma$ metric.
		We create a lower barrier for the superharmonic function $\mathcal{V}$ by using a Bochner-style Laplacian comparison (a technique appearing in \cite{SY}).
		This is available because the curvature of $(\Sigma^2,\tilde{g}_\Sigma)$ is positive: $\widetilde{K}_\Sigma=\frac12\tilde{s}_\Sigma\ge0$ by (\ref{EqnConformalScalar}).
		Therefore comparing the $\widetilde{g}_\Sigma$ Laplacian $\widetilde{\triangle}_\Sigma$ to the Laplacian on flat Euclidean space $\triangle_{Eucl}$, the usual Laplacian comparison for the distance function provides $\widetilde\triangle_\Sigma{r}\le\triangle_{Eucl}r={r}^{-1}$.
		Therefore $\widetilde\triangle_\Sigma{r}^{-1}\ge{r}^{-3}$.
		
		Now we create a barrier for $\mathcal{V}$.
		We consider the annulus $B_p(R)\setminus{}B_p(R/2)$.
		On the outer boundary $\partial{}B_p(R)$ the function $\left(r^{-1}-R^{-1}\right)$ equals 0.
		Because the inner boundary $\partial{}B_p(R/2)$ is compact in $\Sigma^2$, there exists some $\epsilon>0$ so $\mathcal{V}^{\frac12}\ge\epsilon\left(r^{-1}-R^{-1}\right)$.
		Using $\widetilde{\triangle}{}_\Sigma\mathcal{V}^{\frac12}=0$, we now have
		\begin{eqnarray}
			\begin{aligned}
				&\mathcal{V}^{\frac12}\,-\,\epsilon\left(r^{-1}\,-\,R^{-1}\right)
				\;\ge\;0 \quad \text{on $\partial\mathcal{A}$} \quad \text{and} \\
				&\widetilde{\triangle}_\Sigma\left(\mathcal{V}^{\frac12}
				\,-\,\epsilon\left(r^{-1}\,-\,R^{-1}\right)\right)
				\;\le\;-\epsilon{r}^{-3}\;<\;0.
			\end{aligned}
		\end{eqnarray}
		Thus $\mathcal{V}^{\frac12}-\epsilon(r^{-1}-R^{-1})$ is superharmonic on the annulus and non-negative on its boundary.
		So $\mathcal{V}^{\frac12}\ge\epsilon(r^{-1}-R^{-1})$ on the annulus by the maximum principle.
		Because the path $\gamma$ lies within the annulus for $t\in[R/2,R]$ we immediately have
		\begin{eqnarray}
			\begin{aligned}
				\big(\left(\mathcal{V}\circ\gamma\right)(t)\big)^{\frac12}
				&\;>\;\epsilon\left(t^{-1}-R^{-1}\right),
				\quad \text{which is} \\
				\big(\left(\mathcal{V}\circ\gamma\right)(t)\big)^{-\frac14}
				&\;<\;\epsilon^{-\frac12}\left(t^{-1}-R^{-1}\right)^{-\frac12}
				\;=\;\epsilon^{-\frac12}(tR)^{\frac12}\left(R-t\right)^{-\frac12}.
			\end{aligned}
		\end{eqnarray}
		This allows us to estimate the $g_\Sigma$-length of $\gamma$.
		From $|\dot\gamma|_{g_\Sigma}=\mathcal{V}^{-\frac14}|\dot\gamma|_{\tilde{g}_\Sigma}$, we have
		\begin{eqnarray}
			\begin{aligned}
				|\dot\gamma|_{g_\Sigma}
				&\;=\;\big(\mathcal{V}\circ\gamma\big)^{-\frac14}\cdot|\dot\gamma|_{\tilde{g}_\Sigma}
				\;<\;\epsilon^{-\frac12}(tR)^{\frac12}\left(R-t\right)^{-\frac12}\cdot|\dot\gamma|_{\tilde{g}_\Sigma}
			\end{aligned}
		\end{eqnarray}
		within the annulus.
		Because also $|\dot\gamma|_{\tilde{g}_\Sigma}=1$, we have
		\begin{eqnarray}
			\begin{aligned}
				Len_{g_\Sigma}(\gamma[R/2,R])
				&\;=\;\int_{R/2}^R|\dot\gamma|_{g_\Sigma}dt
				\;=\;
				\int_{R/2}^R(\mathcal{V}{}^{-\frac14}\circ\gamma)(t)\,dt \\
				&\quad\;\le\;
				\epsilon^{-\frac12}R^{\frac12}\int_{R/2}^Rt^{\frac12}(R-t)^{-\frac12}dt
				\;=\;
				\epsilon^{\frac12}\frac{2+\pi}{4}R^{\frac32}.
			\end{aligned}
		\end{eqnarray}
		The total length of $\gamma$ in the $g_\Sigma$ metric is $\int_0^{R/2}|\dot\gamma|_{g_\Sigma}dt+\int_{R/2}^R|\dot\gamma|_{g_\Sigma}dt$.
		The first integral is finite because $\gamma(t)$ remains within a compact region of $\Sigma^2$ for $t\in[0,R/2]$.
		Thus $g_\Sigma$-length of $\gamma$ is finite.
		This contradiction completes the proof.
	\end{proof}
	
	\begin{theorem}[{\it cf.} Theorem \ref{ThmComplete}] \label{ThmRTwoClassification}
		Assume $(\Sigma^2,g_\Sigma)$ is geodesically complete and $\triangle_\Sigma\sqrt{\mathcal{V}}\le0$.
		Then $(\Sigma^2,g_\Sigma)$ and $(\Sigma^2,\tilde{g}_\Sigma)$ are flat.
	\end{theorem}
	\begin{proof}
		By Lemma \ref{LemmaConformalComplete}, $(\Sigma^2,\widetilde{g}_\Sigma)$ is complete, and by (\ref{EqnConformalScalar}) $\widetilde{K}_\Sigma=\frac12\tilde{s}_\Sigma\ge0$.
		By the Cheng-Yau criterion for parabolicity, $(\Sigma^2,\widetilde{g}_\Sigma)$ is parabolic so is biholomorphic to $\mathbb{C}$.
		Then $\mathcal{V}^{\frac12}$ is a positive superharmonic function on $\mathbb{C}$, so is constant.
		By Lemma \ref{LemmaTConstFlat}, $g_\Sigma$ is flat and therefore so is $\tilde{g}_\Sigma$.
	\end{proof}
	
	\begin{corollary} [{\it cf.} Theorem \ref{ThmComplete}] \label{CorRTwoClassification}
		Assume $(M^4,J,g,\mathcal{X}^1,\mathcal{X}^2)$ has $s\ge0$, and assume the distribution $\{\mathcal{X}^1,\mathcal{X}^2\}$ is always rank 2.
		Then $M^4$ is flat $\mathbb{C}^2$ and $\mathcal{X}^1$, $\mathcal{X}^2$ are translational Killing fields.
	\end{corollary}
	\begin{proof}
		Since $\mathcal{V}=0$ if and only if $\{\mathcal{X}^1,\mathcal{X}^2\}$ is a linearly dependent set, the reduction $\Sigma^2$ has no edges.
		The moment map is therefore a Riemannian submersion, so the metric polytope $(\Sigma^2,g_\Sigma)$ is complete.
		By equation (\ref{EqnConformalScalar}) we have $\triangle_\Sigma\sqrt{\mathcal{V}}\le0$.
		Theorem \ref{ThmRTwoClassification} now implies $g_\Sigma$ is flat, and $g_\Sigma$ is a constant matrix when expressed in $\varphi^1$-$\varphi^2$ coordinates.
		Referring to (\ref{EqnsGJOmegaM}), we have that $G$ is a constant matrix, so $g$ and $J$ are constant.
		Thus $M^4$ is also flat.
		
		The fact that $M^4\approx\mathbb{C}^2$ comes from our earlier assumption that $M^4$ is simply connected.
		The fact that the Killing fields are translational is due to the fact that they are nowhere zero.
	\end{proof}
	
	{\bf Remark.}
	Crucial to the proofs of Lemma \ref{LemmaTConstFlat} and Theorem \ref{ThmRTwoClassification} is the fact that a complete, simply connected $\Sigma^2$ with $K_\Sigma\ge0$ is biholomorphic to $\mathbb{C}$.
	This is a simple consequence of volume comparison and the Cheng-Yau criterion for parabolicity, namely that $\int_1^\infty\frac{t}{Vol\,B_t}dt=\infty$.
	The assertion that a simply connected, boundaryless, complete Riemann surface is parabolic if and only if it is actually $\mathbb{C}$ is a consequence of the uniformization theorem.
	The subject of parabolicity has received a great deal of attention; for a tiny sampling of this large subject see \cite{LT1} \cite{LT2} \cite{HK} \cite{GM} \cite{Web1} and references therein.

	%
	%
	%
	%
	%
	% Sec 5: The Asymptotic Model Geometries
	%
	%
	%
	%
	%
	\section{The Asymptotic Model Geometries} \label{SecAsymptoticAnalysis}
	Here we prove Theorem \ref{ThmSFClassification}.
	The earlier paper \cite{Web1} showed that \textit{if} the reduction $\Sigma^2$ of $M^4$ is closed and has at least one edge, then we can write down all possible metrics on $M^4$.
	Theorem \ref{ThmComplete} removed the need for the polygon to have an edge.
	We remark that cases (3) and (5) consist of one metric each, and (4) assumes that the polygon is closed.
	Thus to complete the ``$\Rightarrow$'' part of Theorem \ref{ThmSFClassification}, we must show, in cases (1) and (2), that the metrics written down in \cite{Web1} indeed meet the criteria of being ``asymptotically general'' and ``asymptotically equivariantly $\mathbb{R}^2\times\mathbb{S}^2$'' in the sense discussed in the introduction.
	We check this in Lemmas \ref{LemmaGeneralCase} and \ref{LemmaAsymptSRCheck} below.
	
	We must also prove the converse, that if the metric on $M^4$ has the asymptotic structure of any of the cases (1)-(5), then its reduction $\Sigma^2$ is closed.
	Again we must only check this in the cases (1) and (2), as the closure of $\Sigma^2$ is built into cases (3) and (4) already.
	In \ref{SubsecReductions} we prove the following lemma.
	\begin{lemma} \label{LemmaModelKillingBounded}
		Assume $M^4$ is either ``equivariantly asymptotically general,'' or ``equivariantly asymptotically $\mathbb{R}^2\times\mathbb{S}^2$.''
		Then the parent manifold $M^4$ has two unbounded polar submanifolds, and each has a Killing field which remains bounded away from zero.
		By Theorem \ref{ThmClosed} consequently any such manifold $M^4$ has closed reduction.
	\end{lemma}
	
	In cases (1), (2), and (3) we have two different bijections between $\Sigma^2$ and the closed upper half-plane.
	These are
	\begin{equation}
		z:\Sigma^2\rightarrow\overline{H}{}^2 \quad\text{and}\quad
		\Phi:\overline{H}{}^2\rightarrow\Sigma^2.
	\end{equation}
	The first is the ``volumetric normal function'' $z=x+\sqrt{-1}y$ discussed in \S\ref{SubSecRedProps}, which is a biholomorphism by Proposition 1.1 of \cite{Web1}.
	The second is given by $\Phi(x,y)=(\varphi^1(x,y),\varphi^2(x,y))$, which is smooth everywhere except at corner points, where it is Lipschitz.
	As discussed in Section \ref{SecPrelims} the geometry of $M^4$ is completely determined by $(\Sigma^2,g_\Sigma)$, which itself is completely determined by the relationship between the $(\varphi^1,\varphi^2)$ and $(x,y)$ coordinates.
	From (2.13) of \cite{Web1} the metric is
	\begin{equation}
		\begin{aligned}
			&g_\Sigma\;=\;\frac{1}{y}\text{det}(A)\left(dx\otimes{}dx\,+\,dy\otimes{}dy\right)
		\end{aligned} \label{EqnMetricInXY}
	\end{equation}
	where $A$ is the coordinate transition matrix $A=\left(\frac{\partial\{\varphi^1,\varphi^2\}}{\partial\{x,y\}}\right)$.
	
	As discussed in the introduction, we say the manifold $(M^4,g)$ is equivariantly asymptotically modeled on $(\widetilde{M}{}^4,\widetilde{g})$ to order $l$ if, outside some compact set $\widetilde{K}\subset\widetilde{M}{}^4$ there exists an equivariant covering map $\pi:\widetilde{M}{}^4\setminus{}\widetilde{K}\rightarrow{}M^4$ so that
	\begin{equation}
		\left\|\pi^*g\;-\;\widetilde{g}\right\|_{\widetilde{g}}\;=\;O(\tilde{\rho}{}^{-l}), \text{ or $o(1)$ if $l=0$}.
	\end{equation}
	In what follows, the models $\widetilde{M}{}^4$ will always be a toric Kahler manifold with momentum functions $\tilde\varphi^1,\tilde\varphi^2:\overline{H}{}^2\rightarrow\widetilde{\Sigma}{}^2$.
	We can use the same $(x,y)$ coordinates to express either metric, and take the $\pi$ to simply be the identity.
	Then the comparison is
	\begin{equation}
		\begin{aligned}
			\left\|\pi^*g\;-\;\tilde{g}\right\|_{\widetilde{g}}
			&\;=\;
			\frac{1}{y}\left|\text{det}(A)-\text{det}(\widetilde{A})\right|
			\left\|dx\otimes{}dx+dy\otimes{}dy\right\|_{\widetilde{g}} \\
			&\;=\;
			\frac{1}{y}\left|\text{det}(A)-\text{det}(\widetilde{A})\right|
			\sqrt{2}(\|dx\|_{\widetilde{g}})^2 \\
			&\;=\;
			\frac{1}{y}\left|\text{det}(A)-\text{det}(\widetilde{A})\right|
			\sqrt{2}\frac{y}{\text{det}(\widetilde{A})}
			\;=\;
			\sqrt{2}\left|1-\frac{\text{det}(A)}{\text{det}(\widetilde{A})}\right|.
		\end{aligned} \label{EqnTheKeyQuotient}
	\end{equation}
	Therefore proving that $\widetilde{M}{}^4$ is an equivariant asymptotic model for $M^4$ comes down to comparing coordinate transition Jacobians for the two manifolds.
	
	%
	%
	%
	% Subsection 5.1
	%
	%
	%
	\subsection{The ``general'' case.}
	We show that any metric on a polygon with non-parallel rays is asymptotically modeled on a generalized Taub-NUT metric.
	From \cite{Web1}, after possible affine recombination, the momentum functions on such a polygon always have the form
	\begin{equation}
		\small
		\begin{aligned}
			\varphi^1
			&\;=\;\frac{\alpha}{2}(y)^2
			+\frac{s_d}{2}\left((x-x_d)+\sqrt{(x-x_d)^2+(y)^2}\right) \\
			&+\sum_{i=1}^{d-1}\frac{m_{i+1}-m_i}{2(x_{i+1}-x_i)}\left(
			x_{i+1}-
			\sqrt{(x-x_{i+1})^2+(y)^2}
			-x_i+\sqrt{(x-x_{i})^2+(y)^2}
			\right) \\
			\varphi^2
			&\;=\;\frac{\beta}{2}(y)^2
			+n_1
			+\frac{s_0}{2}\left(-(x-x_1)+\sqrt{(x-x_1)^2+(y)^2}\right) \\
			&+\sum_{i=1}^{d-1}\frac{n_{i+1}-n_i}{2(x_{i+1}-x_i)}\left(
			x_{i+1}-\sqrt{(x-x_{i+1})^2+(y)^2}
			-x_i+\sqrt{(x-x_{i})^2+(y)^2}
			\right)
		\end{aligned} \label{EqnCompsOfOrigGeneric}
	\end{equation}
	where the values $s_0$, $s_d$, $m_i$, $n_i$, and the differences $x_{i+1}-x_i$ are determined uniquely from the labeled polygon.
	The values $\alpha,\beta\ge0$ are ``free'' parameters, meaning in this case that changing $\alpha$ or $\beta$ changes the metric but does not change the polygon or its labels.
	Specifically, by the recipe of \cite{Web1}, the values $\{x_1,\dots,x_d\}$ are the Lipschitz points on the $x$-axis in $\mathbb{C}$ and $\{p_1,\dots,p_d\}$ are the vertex point on the boundary of $\Sigma^2$ where $p_i=(m_i,n_i)$.
	The polygon boundary labels are $s_0$ and $s_d$ on the rays, and $s_i=\frac{|p_{x_{i+1}}-p_i|}{x_{i+1}-x_i}$ on the $i^{th}$ boundary segment.
	See Figure \ref{FigGeneralModel}
	\begin{figure}[h]
		\begin{tabular}{ll}
			\includegraphics[scale=0.17]{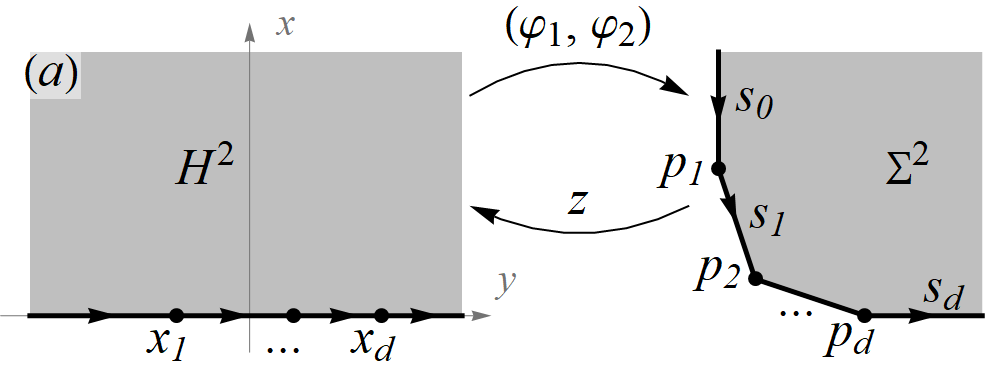}
			& \includegraphics[scale=0.17]{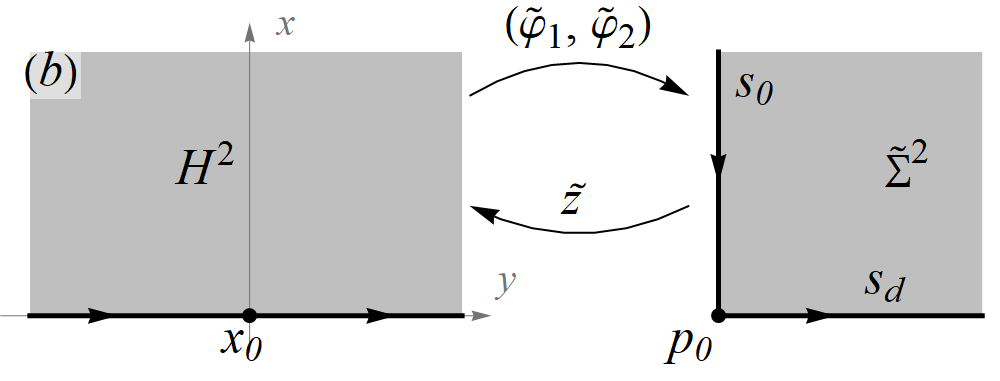}
		\end{tabular}
		\caption{
			(a) A typical ``general case'' polygon $\Sigma^2$ with non-parallel rays, and
			(b) a Taub-NUT style asymptotic model $\widetilde{\Sigma}{}^2$.
		} \label{FigGeneralModel}
	\end{figure}
	
	Our comparison manifold $\widetilde{M}{}^4$ will be any of the generalized Taub-NUT manifolds; these have momentum functions
	\begin{equation}
		\begin{aligned}
			\tilde{\varphi}{}^1
			&\;=\;\frac{\alpha}{2}(y)^2
			+\frac{s_d}{2}\left(x+\sqrt{(x)^2+(y)^2}\right) \\
			\tilde{\varphi}{}^2
			&\;=\;\frac{\beta}{2}(y)^2
			+\frac{s_0}{2}\left(-x+\sqrt{(x)^2+(y)^2}\right). 
		\end{aligned} \label{EqnTaubNUTModels}
	\end{equation}
	The resulting generalized Taub-NUT manifolds were thoroughly studied in \cite{Web3}.
	Their reductions are wedges in $\mathbb{R}^2$, which after possible affine transformation is a quarter-plane.
	We must compute $A=\left(\frac{\partial\{\varphi^1,\varphi^2\}}{\{x,y\}}\right)$ and $\widetilde{A}=\left(\frac{\partial\{\tilde\varphi^1,\tilde\varphi^2\}}{\{x,y\}}\right)$.
	From (\ref{EqnTaubNUTModels}),
	\begin{equation}
		\small
		\begin{aligned}
			\frac{\partial\tilde\varphi^1}{\partial{}x}
			&\;=\;\frac{s_d}{2}\left(1+\cos\theta\right), \;\;\;\,
			\frac{\partial\tilde\varphi^1}{\partial{}y}
			\;=\;\left(\alpha{}r+\frac{s_d}{2}\right)\sin\theta, \\
			\frac{\partial\tilde\varphi^2}{\partial{}x}
			&\;=\;\frac{s_0}{2}\left(-1+\cos\theta\right), \;
			\frac{\partial\tilde\varphi^2}{\partial{}y}
			\;=\;\left(\beta{}r+\frac{s_0}{2}\right)\sin\theta
		\end{aligned} \label{EqnCompsOfCompTaubNUT}
	\end{equation}
	where we used $x=r\cos\theta$ and $y=r\sin\theta$; because $x$ and $y$ are in the upper half-plane we have $\theta\in[0,\pi]$.
	Then we have
	\begin{equation}
		\small
		\begin{aligned}
			&\widetilde{A}
			=\left(\begin{array}{cc}
				0 & \alpha \\
				0 & \beta
			\end{array}\right)y
			+
			\frac12\left(\begin{array}{cc}
				s_d(1+\cos\theta) & s_d\sin\theta \\
				s_0(-1+\cos\theta) & s_0\sin\theta
			\end{array}\right) \\
			&\text{det}(\widetilde{A})
			=
			\frac12\left(\alpha{}s_0+\beta{}s_d-(\alpha{}s_0-\beta{}s_d)\cos\theta\right)y
			+\frac12s_0s_d\sin\theta.
		\end{aligned}
	\end{equation}
	For $A$ and $\text{det}(A)$, we use (\ref{EqnCompsOfOrigGeneric}) to get
	\begin{equation}
		\small
		\begin{aligned}
			\frac{\partial\varphi^1}{\partial{}x}
			&\;=\;
			\frac{\partial\tilde\varphi^1}{\partial{}x}+\frac{s_d}{2}\left(\frac{x-x_d}{\sqrt{(x-x_d)^2+(y)^2}}-\frac{x}{\sqrt{(x)^2+(y)^2}}\right)\\
			&\quad+\frac12\sum_{i=1}^{d-1}\frac{m_{i+1}-m_i}{x_{i+1}-x_i}\left(
			\frac{x-x_{i}}{\sqrt{(x-x_{i})^2+(y)^2}}
			-\frac{x-x_{i+1}}{\sqrt{(x-x_{i+1})^2+(y)^2}}
			\right) \\
			\frac{\partial\varphi^1}{\partial{}y}
			&\;=\;\frac{\partial\tilde\varphi^1}{\partial{}y}
			+\frac{s_d}{2}\left(\frac{y}{\sqrt{(x-x_d)^2+(y)^2}}-\frac{y}{\sqrt{(x)^2+(y)^2}}\right) \\
			&\quad+\frac12\sum_{i=1}^{d-1}\frac{m_{i+1}-m_i}{x_{i+1}-x_i}\left(
			\frac{y}{\sqrt{(x-x_{i})^2+(y)^2}}
			-\frac{y}{\sqrt{(x-x_{i+1})^2+(y)^2}}\right)
		\end{aligned} \label{EqnFourForA1}
	\end{equation}
	and similarly
	\begin{equation}
		\begin{aligned}
			\frac{\partial\varphi^2}{\partial{}x}
			&\;=\;
			\frac{\partial\tilde\varphi^2}{\partial{}x}
			+\frac{s_0}{2}\left(\frac{x-x_1}{\sqrt{(x-x_1)^2+(y)^2}}-\frac{x}{\sqrt{(x)^2+(y)^2}}\right)\\
			&\quad+\frac12\sum_{i=1}^{d-1}\frac{n_{i+1}-n_i}{x_{i+1}-x_i}\left(
			\frac{x-x_{i}}{\sqrt{(x-x_{i})^2+(y)^2}}
			-\frac{x-x_{i+1}}{\sqrt{(x-x_{i+1})^2+(y)^2}}
			\right) \\
			\frac{\partial\varphi^2}{\partial{}y}
			&\;=\;\frac{\partial\tilde\varphi^2}{\partial{}y}
			+\frac{s_d}{2}\left(\frac{y}{\sqrt{(x-x_1)^2+(y)^2}}-\frac{y}{\sqrt{(x)^2+(y)^2}}\right) \\
			&\quad+\frac12\sum_{i=1}^{d-1}\frac{n_{i+1}-n_i}{x_{i+1}-x_i}\left(
			\frac{y}{\sqrt{(x-x_{i})^2+(y)^2}}
			-\frac{y}{\sqrt{(x-x_{i+1})^2+(y)^2}}\right)
		\end{aligned} \label{EqnFourForA2}
	\end{equation}
	We estimate these lengthy expressions using the binomial formula
	\begin{equation}
		\frac{1}{\sqrt{1+t}}\;=\;\sum_{k=0}^\infty{{-1/2}\choose{k}}t^k
		\;=\;1-\frac12t\,+\,\frac38t^2-\frac{5}{16}t^3\pm\dots \label{EqnSeriesDeriv}
	\end{equation}
	which converges when $t\in(-1,1]$.
	The idea is to use $t=\frac{1}{r}$, so series of the form (\ref{EqnSeriesDeriv}) converge for all sufficiently large $r$.
	We have
	\begin{equation}
		\begin{aligned}
			&\frac{x-x_{i}}{\sqrt{(x-x_{i})^2+(y)^2}}
			\;=\;\frac{x-x_i}{r}\frac{1}{\sqrt{1-2\frac{x_i}{r}\cos\theta+\left(\frac{x_i}{r}\right)^2}} \\
			&\quad\;=\;\left(\cos\theta-\frac{x_i}{r}\right)
			\sum_{k=0}^\infty{{-1/2}\choose{k}}\left(-2x_ir^{-1}\cos\theta+(x_i)^2r^{-2}\right)^k \\
			&\quad\;=\;
			\cos\theta
			-\left(x_i\sin^2\theta\right)\frac{1}{r}
			-\frac32\left(x_i^2\cos\theta\sin^2\theta\right)\frac{1}{r^2}
			+O(r^{-3})
		\end{aligned}
	\end{equation}
	and
	\begin{equation}
		\begin{aligned}
			&\frac{y}{\sqrt{(x-x_{i})^2+(y)^2}}
			\;=\;\frac{y}{r}\frac{1}{\sqrt{1-2\frac{x_i}{r}\cos\theta+\left(\frac{x_i}{r}\right)^2}} \\
			&\quad\;=\;\sin\theta
			\sum_{k=0}^\infty{{-1/2}\choose{k}}\left(-2x_ir^{-1}\cos\theta+(x_i)^2r^{-2}\right)^k \\
			&\quad\;=\;
			\sin\theta
			+\left(x_i\sin\theta\cos\theta\right)\frac{1}{r}
			+\frac14(x_i)^2\left(1+3\cos2\theta\right)\sin\theta\frac{1}{r^2}
			+O(r^{-3})
		\end{aligned}
	\end{equation}
	Then we have the series approximation for $A$
	\begin{equation}
		\begin{aligned}
			&A
			=\left(\begin{array}{cc}
				0 & \alpha \\
				0 & \beta
			\end{array}\right)y
			+
			\frac12\left(\begin{array}{cc}
				s_d(1+\cos\theta) & s_d\sin\theta \\
				s_0(-1+\cos\theta) & s_0\sin\theta
			\end{array}\right) \\
			&\quad\quad
			+
			\frac12\left(\begin{array}{cc}
				(m_d-s_dx_d)\sin^2\theta &-(m_d-s_dx_d)\sin\theta\cos\theta \\
				-(n_1+s_0x_1)\sin^2\theta & (n_1+s_0x_1)\sin\theta\cos\theta
			\end{array}\right)\frac{1}{r}
			+O(r^{-2})
		\end{aligned}
	\end{equation}
	and for $\text{det}(A)$ we find
	\begin{equation}
		\begin{aligned}
			&\text{det}(A)
			=\frac12\left((s_0\alpha+s_d\beta)-(s_0\alpha-s_d\beta)\cos\theta\right)y \\
			&\hspace{0.5in}+
			\frac12\Big(
			(s_0s_d+\alpha(n_1+s_0x_1)+\beta(m_d-s_dx_d)) \\
			&\hspace{1.0in}-(\alpha(n_1+s_0x_1)+\beta(m_d-s_dx_d))\cos^2\theta
			\Big)\sin\theta
			+O(r^{-1}).
		\end{aligned}
	\end{equation}
	Then because $\text{det}(A)$ and $\text{det}(\widetilde{A})$ have the same leading term, we have therefore $1-\text{det}(A)/\text{det}(\widetilde{A})=O(r^{-1})$.
	A computation gives
	\begin{equation}
		\begin{aligned}
			1-\frac{\text{det}(A)}{\text{det}(\widetilde{A})}
			\;=\;
			\frac{(\alpha(n_1+s_0x_1)+\beta(m_d-s_dx_d))\sin^2\theta}{-(\alpha{}s_0+\beta{}s_d)+(\alpha{}s_0-\beta{}s_d)\cos\theta}\frac{1}{r}
			+O(r^{-2}).
		\end{aligned} \label{EqnGeneralCaseComparison}
	\end{equation}
	To address the worry that the denominator in the leading term of (\ref{EqnGeneralCaseComparison}) might be zero, we point out that $s_0$ and $s_d$ have the same sign and $\alpha,\beta\ge0$.
	Therefore the denominator can only be zero if $\alpha=0$ or $\beta=0$ or both.
	When $\beta=0$ but $\alpha\ne0$ or when $\alpha=0$ but $\beta\ne0$ (this is the ``exceptional'' case) we have
	\begin{equation}
		1-\frac{\text{det}(A)}{\text{det}(\widetilde{A})}
		\;=\;-\frac{(n_1+s_0x_1)(1+\cos\theta)}{s_0}\frac{1}{r}
		+O(r^{-2}). \label{EqnGeneralCaseSpecialComparison}
	\end{equation}
	When both $\alpha=\beta=0$ the Taub-NUT is simply Euclidean space; the estimate is
	\begin{equation}
		\small
		\begin{aligned}
			&1-\frac{\text{det}(A)}{\text{det}(\widetilde{A})} \\
			&\;=\;
			\frac{-(m_ds_0+s_dn_1)-s_0s_d(x_1-x_d)+(m_ds_0-s_dn_1-s_0s_d(x_1+x_d))\cos\theta}{2s_0s_d}
			\frac{1}{r} \\
			&\hspace{0.2in}+O(r^{-2}).
		\end{aligned} \label{EqnGeneralCaseVerySpecialComparison}
	\end{equation}
	
	\begin{lemma} \label{LemmaGeneralCase}
		Assume $M^4$ has a closed reduction polygon $\Sigma^2$ with two boundary rays that are not parallel.
		Then its asymptotic geometry is ALE, ALF, ALF-like, or exceptional to order $1$.
	\end{lemma}
	\begin{proof}
		The comparison momenta $\tilde\varphi^1$, $\tilde\varphi^2$ produce the generalized Taub-NUT metrics; see \cite{Web3} for detailed information.
		When $\alpha=\beta=0$ the ``Taub-NUT'' is flat Euclidean space (so metrics that are asymptotically equivalent to this are ALE).
		When $\alpha=\beta>0$ then the comparison metric is a multiple of the standard Ricci-flat Taub-NUT (so metrics that are asymptotically equivalent to this are asymptotically ALF).
		When $\alpha=0$ but $\beta\ne0$ or vice-versa then the comparison metric is an ``exceptional Taub-NUT'' (so metrics that are asymptotically equivalent to this are asymptotically exceptional).
		And when $\alpha>0$, $\beta>0$ but $\alpha\ne\beta$ then the comparison is a chiral Taub-NUTs (so metrics that are asymptotically equivalent to this are asymptotically ALF-like).
		
		From (\ref{EqnGeneralCaseComparison}), (\ref{EqnGeneralCaseSpecialComparison}), (\ref{EqnGeneralCaseVerySpecialComparison}) we have that ``general case'' polygon metrics $(\Sigma^2,g_\Sigma)$ are asymptotically equivalent to their corresponding comparison metrics $(\widetilde{\Sigma}{}^2,\tilde{g}_\Sigma)$.
		Using Corollary 3.3 of \cite{Web3} we have that the coordinate distance function $r$ on $\mathbb{C}$ is, to within a constant multiple, close to the distance function $\tilde\rho$; therefore we have fast decay $\|\tilde{g}_\Sigma-g_\Sigma\|=O(\tilde\rho{}^{-1})$.
		
		The estimates (\ref{EqnGeneralCaseComparison}), (\ref{EqnGeneralCaseSpecialComparison}), (\ref{EqnGeneralCaseVerySpecialComparison}) show the$(\Sigma^2,g_\Sigma)$ polygon metric is close to the comparison polygon $(\widetilde{\Sigma}{}^2,\tilde{}g_\Sigma)$, but we must also check that the manifold metric $(M^4,g)$ is asymptotically close to the comparison metric $(\widetilde{M}{}^2,\tilde{g})$.
		By (\ref{EqnsGJOmegaM}) the action $\varphi^1$-$\varphi^2$ components of the metric are orthogonal to the angle $\theta_1$-$\theta_2$ components.
		The action components have already been checked: this part of the metric is precisely the polygon metric.
		The part perpendicular to this we denote $g^\perp$.
		By (\ref{EqnsGJOmegaM}) we have
		\begin{equation}
			\begin{aligned}
				g^\perp
				\;=\;
				G^{ij}d\theta_i\otimes{}d\theta_j.
			\end{aligned}
		\end{equation}
		\begin{equation}
			G^{ij}
			=g(\nabla\varphi^i,\nabla\varphi^j)
			=A^{is}A^{jt}g(dx^i,\,dx^j)
			=\delta_{st}A^{is}A^{jt}|dx|^2
		\end{equation}
		We have that $dx$ and $dy$ are orthogonal, so that $\partial/\partial{}x=|dx|^{-2}\nabla{}x$.
		Then we have
		\begin{equation}
			\begin{aligned}
				g^\perp
				&\;=\;
				|dx|^2
				\delta_{st}(A^{is}d\theta_i)\otimes(A^{jt}d\theta_j) \\
				&\;=\;
				|dx|^2
				\delta_{st}\left(A^{is}J\frac{\partial}{\partial\varphi^i}\right)_\flat\otimes\left(A^{jt}J\frac{\partial}{\partial\varphi^j}\right)_\flat \\
				&\;=\;
				|dx|^2
				\delta_{st}\left(J\frac{\partial}{\partial{}x^i}\right)_\flat\otimes\left(J\frac{\partial}{\partial{}x^j}\right)_\flat \\
				&\;=\;
				|dx|^{-2}
				\left(
				Jdx\otimes{}Jdx
				+Jdy\otimes{}Jdy
				\right) \\
			\end{aligned}
		\end{equation}
		where we used $\frac{\partial}{\partial\theta_i}=(Jd\varphi^i)^\sharp$ and the dualization $d\theta_i=-\left(J\frac{\partial}{\partial\varphi^i}\right)_\flat$.
		But then $g^\perp=g_\Sigma(J\cdot,J\cdot)$ and because $J$ is orthogonal, we have $\|\tilde{g}^{\perp}-g^\perp\|_{\tilde{g}}=\|\tilde{g}_\Sigma-g_\Sigma\|_{\tilde{g}_\Sigma}$ which we proved decays like $\tilde\rho{}^{-1}$.
	\end{proof}

	%
	%
	%
	% Subsection
	%
	%
	%
	\subsection{The asymptotically $\mathbb{R}^2\times\mathbb{S}^2$ metrics} \label{SubSecRS}
	In this section we examine those metric polygons $(\Sigma^2,g_\sigma)$ that have parallel rays on $\partial\Sigma^2$, but no lines.
	We show they are all asymptotically modeled by the class of polygons $(\widetilde{\Sigma}{}^2,g_\Sigma)$ given by the following comparison momentum functions
	\begin{equation}
		\begin{aligned}
			\tilde\varphi^1
			&=\frac12(m_1+m_d)
			-\frac{s_0}{2}(x-x_1)+\frac{s_d}{2}(x-x_d) \\
			&\quad+\left(\frac{s_0}{2}+\frac{m_d-m_1}{2(x_d-x_1)}\right)\sqrt{(x-x_1)^2+(y)^2} \\
			&\quad+\left(\frac{s_d}{2}-\frac{m_d-m_1}{2(x_d-x_1)}\right)\sqrt{(x-x_d)^2+(y)^2}
			+\frac{\alpha}{2}(y)^2 \\
			\tilde\varphi^2
			&=\frac12(n_1+n_d)
			+\frac{n_d-n_1}{x_d-x_1}\sqrt{(x-x_1)^2+(y)^2}
			-\frac{n_d-n_1}{x_d-x_1}\sqrt{(x-x_d)^2+(y)^2}
		\end{aligned} \label{EqnSRCompMomentums}
	\end{equation}
	where $x_1$, $x_d$ are the two Lipschitz points on the half-space boundary, and $p_1=(m_1,n_1)$ and $p_d=(m_d,n_d)$ are the corresponding polygon vertex points.
	For a depiction see Figure~\ref{FigRSModel}.
	From \cite{Web1}, if the reduction $\Sigma^2$ of $M^4$ has parallel rays, then after possible affine transformation of the $\varphi^1$-$\varphi^2$ plane its momentum functions are
	\begin{equation}
		\small
		\begin{aligned}
			\varphi^1
			&=m_1
			+\frac{s_0}{2}\left(-(x-x_1)+\sqrt{(x-x_1)^2+(y)^2}\right)
			+\frac{s_d}{2}\left((x-x_d)+\sqrt{(x-x_d)^2+(y)^2}\right) \\
			&
			+\sum_{i-1}^{d-1}\frac{m_{i+1}-m_i}{2(x_{i+1}-x_i)}\left(x_{i+1}-x_i+\sqrt{(x-x_i)^2+(y)^2}-\sqrt{(x-x_{i+1})^2+(y)^2}\right)
			+\frac{\alpha}{2}(y)^2 \\
			\varphi^1
			&=n_1
			+\sum_{i=1}^{d-1}\frac{n_{i+1}-n_i}{2(x_{i+1}-x_i)}\left(x_{i+1}-x_i+\sqrt{(x-x_i)^2+(y)^2}-\sqrt{(x-x_{i+1})^2+(y)^2}\right)
		\end{aligned}
	\end{equation}
	where $x_1,\dots,x_d$ are the Lipschitz points on the upper half-space boundary and $p_i=(m_i,n_i)$ are the corresponding polygon vertex points.
	The parameter $\alpha\ge0$ is a ``free'' parameter, meaning changing this parameter does not change the polygon $\Sigma^2$, but does change the metric $g_\Sigma$.
	\begin{figure}[h]
		\begin{tabular}{ll}
			\includegraphics[scale=0.17]{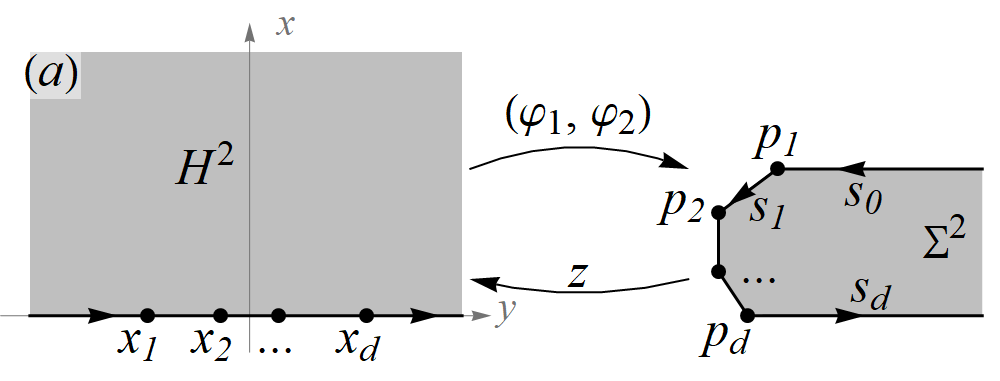}
			& \includegraphics[scale=0.17]{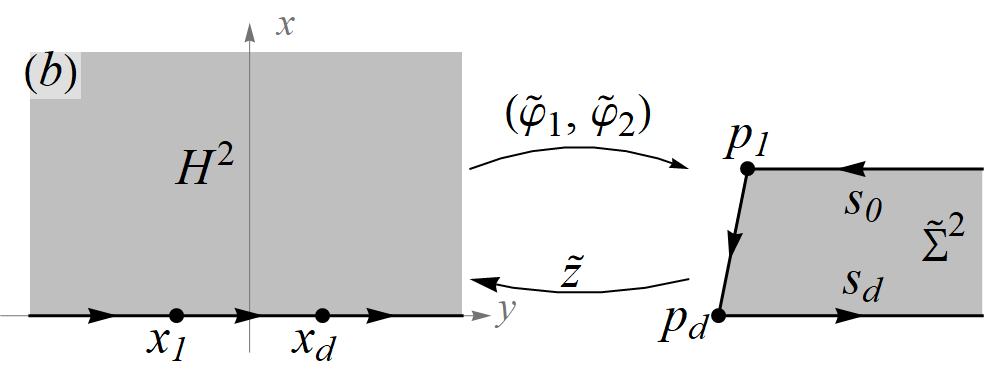}
		\end{tabular}
		\caption{
			(a) Depiction of a typical polygon $\Sigma^2$ with parallel rays, and (b) the corresponding half-strip model $\widetilde{\Sigma}{}^2$.
		} \label{FigRSModel}
	\end{figure}
	As in the ``generic'' case above, we must compute the transition matrix $\widetilde{A}$ for the model polygon and $A$ for the real polygon.
	Then using (\ref{EqnTheKeyQuotient}) we compare the metrics by estimating $1-\text{det}(A)/\text{det}(\widetilde{A})$.
	
	The comparison functions are more complicated than they were in the generic case, but we follow the same process for analyzing them.
	After taking partial derivatives, we approximate the resulting terms using the expansion (\ref{EqnSeriesDeriv}).
	For $\widetilde{A}$
	\begin{equation}
		\small
		\begin{aligned}
			&\widetilde{A}
			\;=\;\left(
			\begin{array}{cc}
				\frac{\partial\tilde\varphi^1}{\partial{}x} & \frac{\partial\tilde\varphi^1}{\partial{}y} \\
				\frac{\partial\tilde\varphi^2}{\partial{}x} & \frac{\partial\tilde\varphi^2}{\partial{}y}
			\end{array}
			\right) \\
			&=\left(
			\begin{array}{cc}
				0 & \alpha\,\sin\theta \\
				0 & 0
			\end{array}\right)\,r
			+\frac12\left(\begin{array}{cc}
				-(s_0+s_d)+(s_0+s_d)\cos\theta & (s_0+s_d)\sin\theta \\
				0 & 0
			\end{array}\right) \\
			&+\left(\begin{array}{cc}
				-\frac12(m_1-m_d+s_0x_1+s_dx_d)\sin^2\theta & \frac12(m_1-m_d+s_0x_1+s_dx_d)\sin\theta\cos\theta \\
				-\frac12(n_1-n_d)\sin^2\theta & \frac12(n_1-n_d)\sin\theta\cos\theta
			\end{array}\right)\frac{1}{r} \\
			&+O(r^{-2}).
		\end{aligned}
	\end{equation}
	and therefore
	\begin{equation}
		\begin{aligned}
			\text{det}(\widetilde{A})
			&\;=\;
			\frac12(n_1-n_d)\alpha\sin^3\theta
			\;+\;O(r^{-1})
		\end{aligned}
	\end{equation}
	The expressions for the $(\Sigma^2,g_\Sigma)$ metric can be arrived at similarly, but even the first few terms are so large they seeem inexpressible except at unreasonable length.
	We don't record them, but fortunately the exact expressions are unnecessary, and we need only record the fact that $\text{det}(A)$ has the same constant term that $\text{det}(\widetilde{A})$ has.
	We have
	\begin{equation}
		\begin{aligned}
			\text{det}(A)
			&\;=\;
			\frac12(n_1-n_d)\alpha\sin^3\theta
			+O(r^{-1})
		\end{aligned} \label{EqnFirstTermOfRSDet}
	\end{equation}
	Because $\text{det}(\widetilde{A})$ and $\text{det}(A)$ expansions have the same constant terms we have
	\begin{equation}
		1-\frac{\text{det}(A)}{\text{det}(\widetilde{A})}
		=O(r^{-1}) \label{EqnRSMetricComps}
	\end{equation}
	However unlike the ``general case'' above, we do not have any known relationship between the function $r=\sqrt{(x)^2+(y)^2}$ on the upper half-space and any distance function on $(\widetilde{M}{}^4,\tilde{g})$.
	But because $r$ is an exhaustion function (its sub-level are compact and their union equals $\widetilde{\Sigma}{}^2$) so we know at least that $r\rightarrow\infty$ as $\tilde\rho\rightarrow\infty$.
	Therefore $O(r^{-1})=o(1)$ on $\widetilde{M}{}^4$.
	This is enough to establish the following theorem.
	\begin{lemma} \label{LemmaAsymptSRCheck}
		Assume $M^4$ has reduction $\Sigma^2$ so that $\Sigma^2$ has parallel rays on its boundary.
		Then $(M^4,g)$ is asymptotically modeled on $(\widetilde{M},\tilde{g})$ to degree $0$, where $(\widetilde{M},\tilde{g})$ is given by the momentum functions (\ref{EqnSRCompMomentums}).
	\end{lemma}
	\begin{proof}
		This a direct consequence of (\ref{EqnFirstTermOfRSDet}) which by (\ref{EqnTheKeyQuotient}) gives $\|\pi^*g-\tilde{g}\|_{\tilde{g}}=o(1)$.
	\end{proof}

	%
	%
	%
	% Subsection 5.3: Closure
	%
	%
	%
	\subsection{Proof that the reductions are closed} \label{SubsecReductions}
	In this section we prove that if $M^4$ has the equivariant asymptotics of (1) or (2), then indeed its reduction $\Sigma^2$ is closed.
	First we establish a way to compare the Killing fields on $M^4$ its model $\widetilde{M}{}^4$.
	\begin{lemma} \label{EqnComparingPolarsToModel}
		Assume $M^4$ has Killing fields $\mathcal{X}^1$, $\mathcal{X}^2$ and $\widetilde{M}{}^4$ has Killing fields $\widetilde{\mathcal{X}}{}^1$, $\widetilde{\mathcal{X}}{}^2$.
		Assume $M^4$ is asymptotically equivariantly modeled on $\widetilde{M}{}^4$ to order 0.
		Then $\left|\|\mathcal{X}^i\|_g-\|\widetilde{\mathcal{X}}^i\|_{\tilde{g}}\right|=o(\|\widetilde{\mathcal{X}}^i\|_{\tilde{g}})$.
	\end{lemma}
	\begin{proof}
		Let $(\varphi^1,\theta_1,\varphi^2,\theta_2)$ and $(\tilde\varphi^1,\tilde\theta_1,\tilde\varphi^2,\tilde\theta_2)$ be the action-angle coordinates on $M^4$ and $\widetilde{M}{}^4$, respectively.
		Then $\mathcal{X}^i=\frac{\partial}{\partial\theta_i}$, $\widetilde{\mathcal{X}}^i=\frac{\partial}{\partial\tilde\theta_i}$.
		Let $\pi$ be the equivariant diffeomorphism that realizes the asymptotic modeling; then $\pi_*\frac{\partial}{\partial\tilde\theta_i}=\frac{\partial}{\partial\theta_i}$.
		Thus
		\begin{equation}
			\small
			\begin{aligned}
				{\left\|\frac{\partial}{\partial\theta_i}\right\|_g}^2
				&=
				g\left(\frac{\partial}{\partial\theta_i},\frac{\partial}{\partial\theta_i}\right)
				=
				g\left(\pi_*\frac{\partial}{\partial\tilde\theta_i},\pi_*\frac{\partial}{\partial\tilde\theta_i}\right) \\
				&=
				(\pi^*g-\tilde{g})\left(\frac{\partial}{\partial\tilde\theta_i},\frac{\partial}{\partial\tilde\theta_i}\right)
				+\tilde{g}\left(\frac{\partial}{\partial\tilde\theta_i},\frac{\partial}{\partial\tilde\theta_i}\right)
			\end{aligned}
		\end{equation}
		so therefore
		\begin{equation}
			\begin{aligned}
				\left|{\left\|\frac{\partial}{\partial\tilde\theta_i}\right\|_{\tilde{g}}}^2
				-{\left\|\frac{\partial}{\partial\theta_i}\right\|_{g}}^2\right|
				\;\le\;\left|(\pi^*g-\tilde{g})\left(\frac{\partial}{\partial\tilde\theta_i},\frac{\partial}{\partial\tilde\theta_i}\right)\right|
				\;\le\;\left\|(\pi^*g-\tilde{g})\right\|_{\tilde{g}}
				{\left\|\frac{\partial}{\partial\tilde\theta_i}\right\|_{\tilde{g}}}^2
			\end{aligned}
		\end{equation}
		which simplifies to
		\begin{equation}
			\begin{aligned}
				&\left|\left\|\frac{\partial}{\partial\tilde\theta_i}\right\|_{\tilde{g}}
				-\left\|\frac{\partial}{\partial\theta_i}\right\|_{g}\right|
				\;\le\;
				\left\|(\pi^*g-\tilde{g})\right\|_{\tilde{g}}
				\frac{\left\|\frac{\partial}{\partial\tilde\theta_i}\right\|_{\tilde{g}}}
				{1+
					\frac{\left\|\frac{\partial}{\partial\theta_i}\right\|_{g}}
					{\left\|\frac{\partial}{\partial\tilde\theta_i}\right\|_{\tilde{g}}}
				}
				\;\le\;
				\left\|(\pi^*g-\tilde{g})\right\|_{\tilde{g}}
				\left\|\frac{\partial}{\partial\tilde\theta_i}\right\|_{\tilde{g}}.
			\end{aligned}
		\end{equation}
		Because $\left\|(\pi^*g-\tilde{g})\right\|_{\tilde{g}}=o(1)$ by assumption, the result follows.
	\end{proof}
	Next we prove asymptotic structure results for two types of model.
	\begin{lemma} \label{LemmaModeledOnTaubNUT}
		Let $\widetilde{M}{}^4$ be one of the generalized Taub-NUT metrics given by (\ref{EqnTaubNUTModels}).
		Then $\widetilde{M}{}^4$ has two unbounded polar submanifolds, and the Killing field on each is asymptotically bounded away from zero.
	\end{lemma}
	\begin{proof}
		For the Taub-NUT momentum functions (\ref{EqnTaubNUTModels}), the polygon $\widetilde{\Sigma}{}^2$ is the first quadrant in the $\tilde\varphi^1$-$\tilde\varphi^2$ plane.
		The map $(\tilde\varphi^1,\tilde\varphi^2):\overline{H}{}^2\rightarrow\widetilde{\Sigma}{}^2$ has just a single Lipschitz point $x_0=0$ on the upper half space boundary, and $M^4$ has two polar submanifolds: one where $y=0$ and $x<0$ (which projects to the positive $\tilde\varphi^2$-axis in the $\tilde\varphi^1$-$\tilde\varphi^2$ plane) and one where $y=0$ and $x>0$ (which projects to the positive $\tilde\varphi^1$-axis).
		
		To measure the Killing fields we use the metric (\ref{EqnMetricInXY})
		We have $d\tilde\varphi^i=J\widetilde{\mathcal{X}}^i$, so
		\begin{equation}
			\small
			\|\widetilde{\mathcal{X}}^i\|^2
			=\|d\tilde\varphi^i\|^2
			=
			\left(\left(\frac{\partial\tilde\varphi^i}{\partial{}x}\right)^2
			+\left(\frac{\partial\tilde\varphi^i}{\partial{}y}\right)^2\right)\|dx\|_{\tilde{g}}^2
			=\left((\widetilde{A}^{i1})^2+(\widetilde{A}^{i2})^2\right)\|dx\|_{\tilde{g}}^2.
		\end{equation}
		From (\ref{EqnMetricInXY}) we have $\|dx\|^2_{\tilde{g}}=y/\text{det}(\widetilde{A})$.
		From (\ref{EqnCompsOfCompTaubNUT}), when $y=0$ we have
		\begin{equation}
			\begin{aligned}
				&\widetilde{A}=
				\left(\begin{array}{cc}
					0 & 0 \\
					-s_0 & 0
				\end{array}\right)\text{when $x<0$ and }
				\left(\begin{array}{cc}
					s_d & 0 \\
					0 & 0
				\end{array}\right)\text{when $x>0$, and} \\
				&\lim_{y\rightarrow0}\frac{1}{y}\text{det}(\widetilde{A})=s_0\left(-\frac{s_d}{2x}+\alpha\right)\text{when $x<0$ and }
				s_d\left(\frac{s_0}{2x}+\beta\right)\text{when $x>0$.}
			\end{aligned}
		\end{equation}
		The positive $\varphi^2$-axis is the projection of the polar submanifold $x<0$, $y=0$ and the Killing field on this submanifold has norm
		\begin{equation}
			\|\widetilde{\mathcal{X}}^2\|_{\tilde{g}}^2
			=\frac{-2x}{s_d-2\alpha{}x}
			\;\stackrel{x\rightarrow\infty}{\longrightarrow}\;\frac{1}{\alpha}.
		\end{equation}
		The positive $\varphi^1$-axis is the projection of the polar submanifold $x>0$, $y=0$ and the Killing field on this submanifold has norm
		\begin{equation}
			\|\widetilde{\mathcal{X}}^2\|_{\tilde{g}}^2
			=\frac{2x}{s_0+2\beta{}x}
			\;\stackrel{x\rightarrow\infty}{\longrightarrow}\;\frac{1}{\beta}.
		\end{equation}
		If $\alpha=0$ or $\beta=0$ we have unbounded growth.
		In any case we have bounds away from zero.
	\end{proof}
	\begin{lemma} \label{LemmaModeledOnRS}
		Let $\widetilde{M}{}^4$ be one of the $\mathbb{R}^2\times\mathbb{S}^2$ model metrics given by the momentum functions of (\ref{EqnSRCompMomentums}).
		Then $\widetilde{M}{}^4$ has two unbounded polar submanifolds, and the Killing field on each grows unboundedly.
	\end{lemma}
	\begin{proof}
		This proof is almost the same as the proof of Lemma \ref{LemmaModeledOnTaubNUT}.
		Again polar submanifolds occur exactly where $y=0$.
		The comparison polygon, as depicted in Figure~\ref{FigRSModel}(b), has its terminal rays parallel to the $\varphi^1$-axis; therefore the Killing field along either unbounded polar submanifold is $\frac{\partial}{\partial\tilde\theta_1}$.
		Again using (\ref{EqnMetricInXY}) we have
		\begin{equation}
			\left\|\frac{\partial}{\partial\tilde\theta_1}\right\|_{\tilde{g}}^2
			\;=\;
			\left\|d\tilde\varphi^1\right\|_{\tilde{g}}^2
			\;=\;
			\frac{y}{\text{det}(\widetilde{A})}
			\left(\left(\frac{\partial\tilde\varphi^1}{\partial{}x}\right)^2+\left(\frac{\partial\tilde\varphi^1}{\partial{}y}\right)^2\right).
		\end{equation}
		The model momenta (\ref{EqnSRCompMomentums}) allow exact computation of this quantity.
		When $y=0$, we compute to the left of the first Lipschitz point (when $x<x_1$) that
		\begin{equation}
			\left(\frac{\partial\tilde\varphi^1}{\partial{}x}\right)^2+\left(\frac{\partial\tilde\varphi^1}{\partial{}y}\right)^2
			=s_0{}^2 \quad\text{and}\quad
			\lim_{y\rightarrow0}\frac{y}{\text{det}(\widetilde{A})}=2\frac{(x-x_1)(x-x_d)}{(n_1-n_d)s_0},
		\end{equation}
		and to the right of the second Lipschitz point (when $x>x_d$) that
		\begin{equation}
			\left(\frac{\partial\tilde\varphi^1}{\partial{}x}\right)^2+\left(\frac{\partial\tilde\varphi^1}{\partial{}y}\right)^2
			=s_d{}^2 \quad\text{and}\quad
			\lim_{y\rightarrow0}\frac{y}{\text{det}(\widetilde{A})}=2\frac{(x-x_1)(x-x_d)}{(n_1-n_d)s_0}.
		\end{equation}
		Therefore in both cases we obtain growth of $\frac{\partial}{\partial\theta_1}$, approximately linearly with $x$.
		Specifically, on the polygon boundary where $y=0$, 
		\begin{equation}
			\left\|\frac{\partial}{\partial\tilde\theta_1}\right\|_{\tilde{g}}
			=\begin{cases}
				\sqrt{\frac{2s_0}{n_1-n_d}}\sqrt{(x-x_1)(x-x_d)}, & x<x_1 \\
				\sqrt{\frac{2s_d}{n_1-n_d}}\sqrt{(x-x_1)(x-x_d)}, & x>x_d.
			\end{cases}
		\end{equation}
		Currenty there is no known relationship between $x$ and the distance function on the manifold, but certainly $x$ is an exhaustion function so $x\rightarrow\infty$ implies $\tilde\rho\rightarrow\infty$.
		Therefore $\|\widetilde{\mathcal{X}}{}^1\|_{\tilde{g}}\rightarrow\infty$ grows unboundedly along either of the unbounded polar submanifolds.
	\end{proof}
	
	\underline{\it Proof of Lemma \ref{LemmaModelKillingBounded}}.
	Assume $M^4$ is equivariantly asymptotic to $\widetilde{M}{}^4$ where $\widetilde{M}{}^4$ is a generalized Taub-NUT.
	Because the diffeomorphism $\pi:\widetilde{M}{}^4\setminus\widetilde{K}\rightarrow{}M^4$ is equivariant the polar submanifolds map to polar submanifolds (as polar submanifolds are zero-sets of Killing fields).
	By Lemma \ref{LemmaModeledOnTaubNUT}, the model manifold $\widetilde{M}^4$ has polar submanifolds with Killing fields bounded from below; therefore by Lemma \ref{EqnComparingPolarsToModel} so does $M^4$.
	Therefore by Theorem \ref{ThmClosed} the reduction $\Sigma^2$ of $M^4$ is closed.
	
	Using Lemma \ref{LemmaModeledOnRS}, if $M^4$ is equivariantly asymptotically $\mathbb{R}^2\times\mathbb{S}^2$, we reach the same conclusion.
	\qed

	%
	%
	%
	% SubSection 5.4: Asymptotics
	%
	%
	%
	\subsection{Proof of Theorem \ref{ThmSFClassification}}
	
	Assume $(M^4,g,J,\mathcal{X}^1,\mathcal{X}^2)$ is a scalar-flat K\"ahler manifold with commuting holomorphic Kililng fields $\{\mathcal{X}^1,\mathcal{X}^2\}$.
	For the ``$\Rightarrow$'' direction, assume its reduction $\Sigma^2$ is closed.
	If $\Sigma^2$ is compact, then passing to $\Sigma^2$ the fact that $s=0$ on $M^4$ means $\triangle_\Sigma\mathcal{V}^{1/2}=0$ on the precompact domain $Int(\Sigma^2)$.
	By the maximum principle $\mathcal{V}^{1/2}$ has no interior maximum, but also $\mathcal{V}^{1/2}=0$ on $\partial\Sigma^2$, so $\mathcal{V}$ is zero.
	This means $\mathcal{X}^1$, $\mathcal{X}^2$ are everywhere linearly dependent, violating the assumption that the manifold is toric.
	
	If $\Sigma^2$ is $\mathbb{R}^2$, Theorem \ref{ThmComplete} states $M^4$ is flat.
	If $\partial\Sigma^2$ has a single line then $\Sigma^2$ is a half-plane, so by Theorem 1.5 of \cite{Web1} then $M^4$ is an exceptional half-plane instanton (or is flat).
	If $\partial\Sigma^2$ has non-parallel rays then Lemma \ref{LemmaGeneralCase} states $M^4$ is equivariantly asymptotically ``General'' and if $\partial\Sigma^2$ has parallel rays then Lemma \ref{LemmaAsymptSRCheck} states $M^4$ is equivariantly asymptotically $\mathbb{R}^2\times\mathbb{S}^2$.
	These four possibilities, along with the case that $\Sigma^2$ is the closed strip, exhausts the possiblities for a closed reduction $\Sigma^2$.
	
	For the ``$\Leftarrow$'' part of the proof, assume $M^4$ has the asymptotic models given by any of (1)-(5) of Theorem \ref{ThmSFClassification}.
	If $M^4$ is ``equivariantly asymptotically general'' meaning it has the asymptotics of (1) or ``equivariantly asymptotically $\mathbb{R}^2\times\mathbb{S}^2$'' meaning it has the asymptotics of (2), the Lemma \ref{LemmaModelKillingBounded} states its reduction $\Sigma^2$ is closed.
	The other three cases are \textit{a fortiori} true: if $M^4$ is the exceptional half-plane instanton or if the reduction of $M^4$ is the closed half-strip then clearly $\Sigma^2$ is closed.
	Lastly if $M^4$ is flat $\mathbb{C}^2$ then it has three possible reductions depending on which commuting Killing fields are chosen.
	If two translational fields are chose then $\Sigma^2$ is $\mathbb{R}^2$, if one translational and one rotational field is chosen then $\Sigma^2$ is a half-plane in $\mathbb{R}^2$, and if two rotational fields are chosen then $\Sigma^2$ is a quarter-plane in $\mathbb{R}^2$.
	This concludes the proof of Theorem \ref{ThmSFClassification}.

	%
	%
	%
	%
	%
	% Section 6: Examples
	%
	%
	%
	%
	%
	\section{Examples} \label{SecExamples}
	
	We give two examples.
	The first is a complete 4-manifold with precompact but non-polygonal reduction.
	The second is the case of the scalar-flat hyperbolic space crossed with a sphere; this is the main example of the ``equivariantly asymptotically $\mathbb{R}^2\times\mathbb{S}^2$'' model geometry.
	\begin{figure}[h]
		\begin{tabular}{ll}
			\includegraphics[scale=0.35]{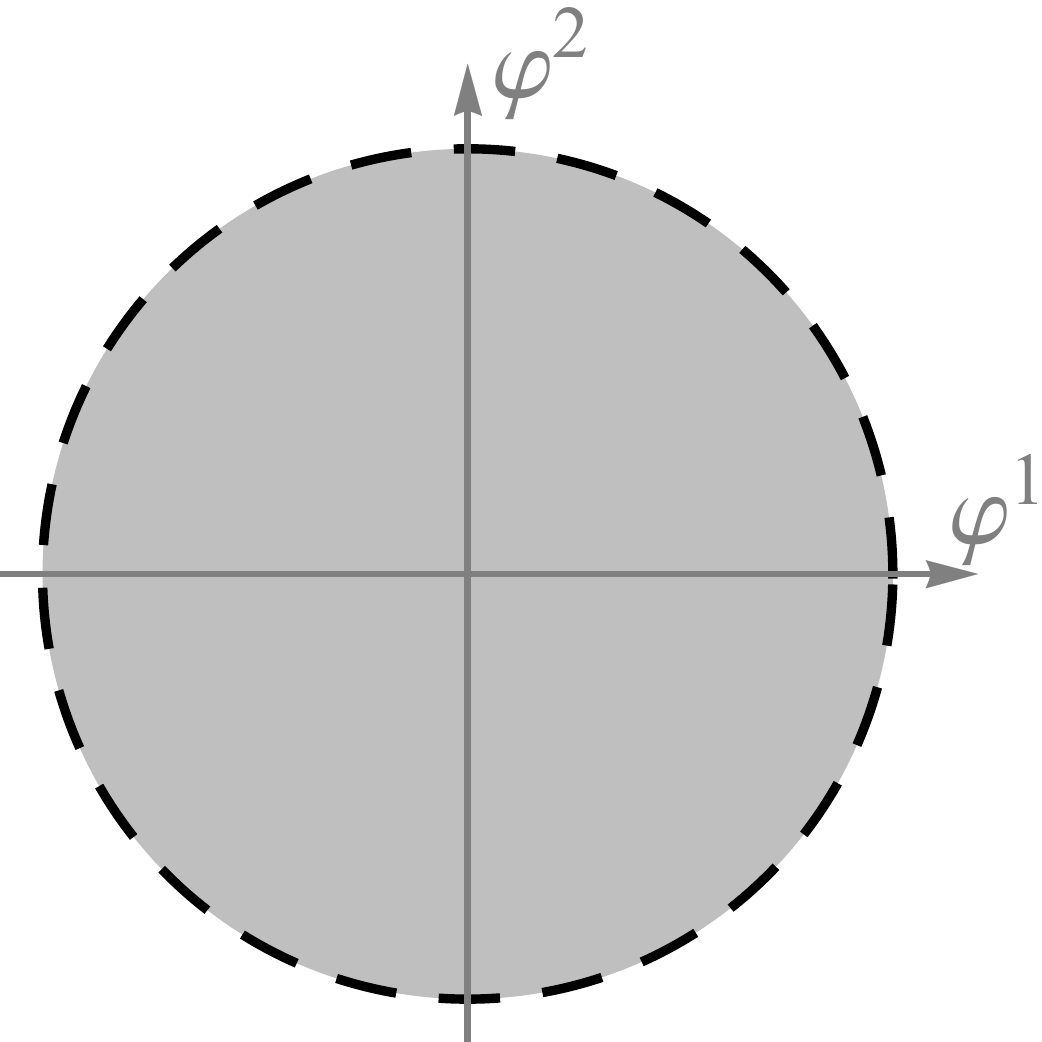}
			& \hspace{0.4in}\includegraphics[scale=0.52]{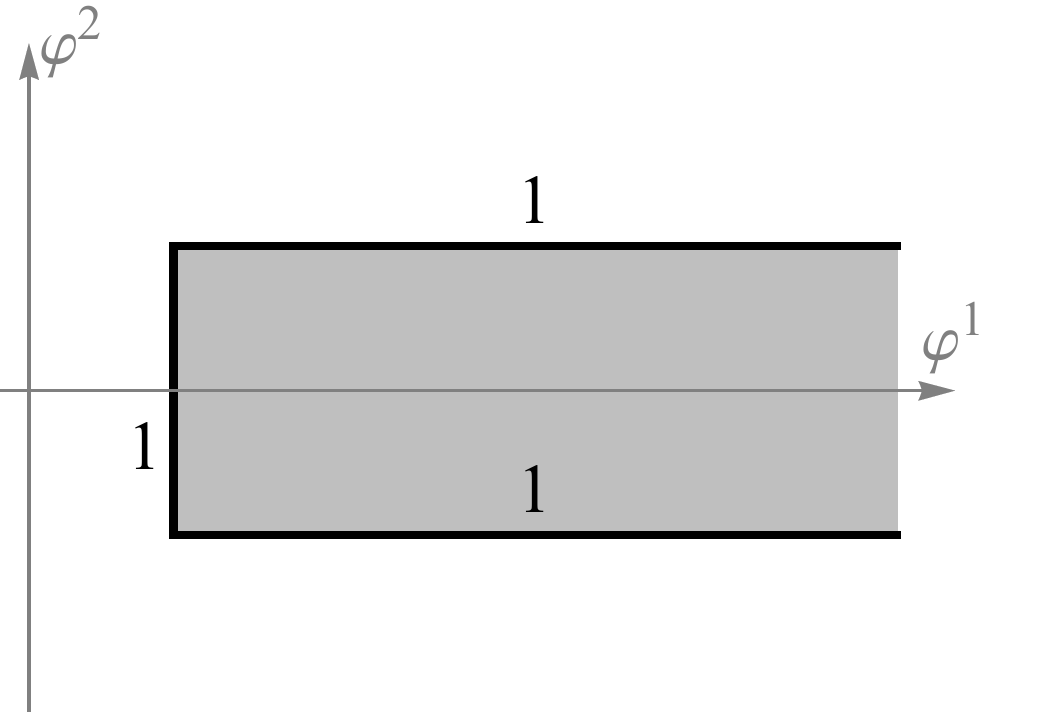}
		\end{tabular}
		\caption{
			(a) A non-polygon reduction, and (b) the reduction of the cross product of hyperbolic space with the unit sphere.
		}
		\label{FigFinal}
	\end{figure}

	%
	%
	%
	% Subsection 6.1: non-polygon example
	%
	%
	%
	\subsection{A non-polygon example} \label{SubsubsecNonpolygon}
	
	Consider the function
	\begin{eqnarray}
		G(\varphi^1,\varphi^2)\;=\;-\frac12\log\left(1-\left(\varphi^1\right)^2-\left(\varphi^2\right)^2 \right) \label{EqnDefOfUu}
	\end{eqnarray}
	defined on the open disk $\Sigma^2=\{\left(\varphi^1\right)^2+\left(\varphi^2\right)^2<1\}$.
	Interpreting this as a symplectic potential, we obtain metric $g_{\Sigma,ij}=G_{ij}d\varphi^i\otimes{}d\varphi^j$ where $G_{ij}\triangleq\frac{\partial^2{}G}{\partial\varphi^i\partial\varphi^j}$.
	Explicitly,
	\begin{eqnarray}
		G_{ij}
		\;=\;
		\left(\begin{array}{cc}
			\frac{1+\left(\varphi^1\right)^2-\left(\varphi^2\right)^2}{\left(1-\left(\varphi^1\right)^2-\left(\varphi^2\right)^2\right)^2} & \frac{2\varphi^1\varphi^2}{\left(1-\left(\varphi^1\right)^2-\left(\varphi^2\right)^2\right)^2} \\
			\frac{2\varphi^1\varphi^2}{\left(1-\left(\varphi^1\right)^2-\left(\varphi^2\right)^2\right)^2} & \frac{1-\left(\varphi^1\right)^2+\left(\varphi^2\right)^2}{\left(1-\left(\varphi^1\right)^2-\left(\varphi^2\right)^2\right)^2}
		\end{array}\right).
	\end{eqnarray}
	Then $(\Sigma^2,g_{\Sigma})$ is a surface of revolution.
	See Figure \ref{FigFinal}(a).
	Considering a path of the form $\gamma(t)=(a\,t,\;b\,t)$ where $a^2+b^2=1$, then $|\dot\gamma|=\sqrt{2}\frac{\sqrt{1+t^2}}{1-t^2}$ and we see that $\int_0^1|\dot\gamma|dt=\infty$, so this surface is complete.
	By \ref{EqnsGJOmegaM}, then $(\Sigma^2,g_\Sigma)$ is the Arnold-Liouville reduction of a complete toric K\"ahler 4-manifold on $\Sigma^2\times\mathbb{T}^2$ whose metric is $g=G_{ij}d\varphi^i\otimes{}d\varphi^j+G^{ij}d\theta_i\otimes{}d\theta_j$.
	By the usual formula for the curvature tensor in holomorphic coordinates,
	\begin{equation}
		\Omega^i_j=\sqrt{-1}\partial\left(h^{i\bar{s}}\bar\partial{}h_{j\bar{s}}\right).
	\end{equation}
	Using the complex coordinates of (\ref{EqnCxCoords}) we have $h^{i\bar{\jmath}}=\frac12G^{ij}$ and $\frac{\partial}{\partial{}z_i}=\frac12\nabla\varphi^i$, so
	\begin{equation}
		\Omega^i_j=\frac14\sqrt{-1}\nabla\varphi^k(G^{is}\nabla\varphi^l(G_{js}))\;dz_k\wedge{}d\bar{z}_l. \label{EqnCurvature}
	\end{equation}
	We can compute scalar curvature $s$ either using the Abreu equation (Equation (10) of \cite{Ab1}), or else using (\ref{EqnCurvature}), to find that
	\begin{equation}
		R\;=\;4\frac{2-4(\boldsymbol{\varphi})^2-3(\boldsymbol{\varphi})^4-(\boldsymbol{\varphi})^6}{(1+(\boldsymbol{\varphi})^2)^3}
	\end{equation}
	where we abbreviated $\boldsymbol{\varphi}=\sqrt{(\varphi^1)^2+(\varphi^2)^2}$.
	We see $R$ is not signed; for example $R=+8$ at $\boldsymbol{\varphi}=0$ and decreases to $-3$ asymptotically as $\boldsymbol{\varphi}\rightarrow1$.
	A tedious but straightforward check shows the norm of curvature is bounded.
	We find
	\begin{equation}
		\small
		*(\Omega^i_j\wedge*\Omega^j_i)
		=32\frac{3-10({\boldsymbol{\varphi}})^2+29({\boldsymbol{\varphi}})^4+24({\boldsymbol{\varphi}})^6+19({\boldsymbol{\varphi}})^8+6({\boldsymbol{\varphi}})^{10}+({\boldsymbol{\varphi}})^{12}}{(1+({\boldsymbol{\varphi}})^2)^6}.
	\end{equation}

	%
	%
	%
	% Subsection 6.1: non-polygon example
	%
	%
	%
	\subsection{The example of hyperbolic space crossed with the sphere}
	A typical scalar-flat metric on $\mathbb{R}^2\times\mathbb{S}^2$ is
	\begin{equation}
		g
		=(dr^1)^2+(\sinh(r^1)d\theta_1)^2
		+(dr^2)^2+(\sin(r^2)d\theta_2)^2
	\end{equation}
	where $r^1\in[0,\infty)$, $r^2\in[0,\pi]$, and $\theta_1,\theta_2\in[0,2\pi)$.
	The complex structure is $J\frac{\partial}{\partial\theta_1}=-\sinh(r_1)\nabla{}r^1$ and $J\frac{\partial}{\partial\theta_2}=-\sin(r^2)\nabla{}r^2$; therefore the momentum functions are $\varphi^1=\cosh(r^1)$, $\varphi^2=-\cos(r^2)$.
	In action-angle coordinates we find
	\begin{equation}
		g
		=\frac{1}{(\varphi^1)^2-1}(d\varphi^1)^2+\big((\varphi^1)^2-1\big)(d\theta^1)^2
		+\frac{1}{1-(\varphi^2)^2}(d\varphi^2)^2+\big(1-(\varphi^2)^2\big)(d\theta^2)^2
	\end{equation}
	with ranges $\varphi^1\in[1,\infty)$ and $\varphi^2\in[-1,1]$.
	Therefore $\Sigma^2$ is a closed half-strip with vertex points at $p_1=(1,1)$ and $p_d=(1,-1)$; see Figure~\ref{FigFinal}(b).
	The $\Sigma^2$ metric is
	\begin{equation}
		g_\Sigma\;=\;\frac{1}{(\varphi^1)^2-1}(d\varphi^1)^2+\frac{1}{1-(\varphi^2)^2}(d\varphi^2)^2
	\end{equation}
	which is a flat metric on the closed half-strip $\varphi^1\in[1,\infty)$ and $\varphi^2\in[-1,1]$.
	To compute the volumetric normal coordinates, it is immediate that
	\begin{equation}
		y\;=\;\sqrt{\mathcal{V}}\;=\;\sqrt{((\varphi^1)^2-1)(1-(\varphi^2)^2)},
	\end{equation}
	and then because $J_\Sigma{}d\varphi^1=\sqrt{\frac{1-(\varphi^2)^2}{(\varphi^1)^2-1}}d\varphi^2$ we find
	\begin{equation}
		x\;=\;\varphi^1\varphi^2.
	\end{equation}
	Solving for $\varphi^1$ and $\varphi^2$ in terms of $x$ and $y$, we find
	\begin{equation}
		\begin{aligned}
			&\varphi^1=\frac12\left(\sqrt{(x-1)^2+(y)^2}+\sqrt{(x+1)^2+(y)^2}\right), \\
			&\varphi^2=\frac12\left(-\sqrt{(x-1)^2+(y)^2}+\sqrt{(x+1)^2+(y)^2}\right).
		\end{aligned}
	\end{equation}
	These are precisely the comparison momentum functions of $(\tilde\varphi^1,\tilde\varphi^2)$ of (\ref{EqnSRCompMomentums}) for $s_0=1$, $s_d=1$, $x_1=-1$, $x_d=1$, $p_1=(m_1,n_1)=(1,1)$, $p_d=(m_d,n_d)=(1,-1)$.


\begin{thebibliography}{9}
		
		\bibitem{Ab1}{M. Abreu}, \emph{K\"ahler geometry of toric varieties and extremal metrics}, {International Journal of Mathematics}, {\bf 9} (1998) {641--651}
		
		\bibitem{AS}{M. Abreu and R. Sena-Dias}, \emph{Scalar-flat K\"ahler metrics on non-compact symplectic toric 4-manifolds}, {Annals of Global Analysis and Geometry}, {\bf 41} No. 2 (2012) {209--239}
		
		\bibitem{Ar}{V. Arnold}, \emph{Mathematical Methods of Classical Mechanics}, {Graduate Texts in Mathematics}, {\bf 60}, {Springer-Verlag}, New York, 1978
		
		\bibitem{CDG}{D. Calderbank, L. David, and P. Gauduchon}, \emph{The Guillemin formula and K\"ahler metrics on toric symplectic manifolds}, {Journal of Symplectic Geometry} {\bf 4} No. 1 (2002) {767--784}
		
		\bibitem{CG86}{J. Cheeger and M. Gromov}, \emph{Collapsing Riemannian manifolds while keeping their curvature bounded. I.} Journal of Differential Geometry {\bf{23}} No. 3 (1986)
		
		\bibitem{CY}{S.Y. Cheng and S.T. Yau}, \emph{Differential equations on Riemannian manifolds and their geometric applications}, {Communication on Pure and Applied Mathematics}, {\bf 28} (1975) {333--354}
		
		\bibitem{CK}{S. Cherkis and A. Kapustin}, \emph{Hyperk\"ahler metrics from periodic monopoles}, {Physics Letters D} {\bf 65} (2002) 084015
		
		\bibitem{Delzant}{T. Delzant}, \emph{Hamiltoniens p\'eriodique et images convexes de l'application moment}, {Bulletin de la Soci\'et\'e Math\'ematique de France}, {\bf 116} (1988) {315--339}
		
		\bibitem{Do1}{S. Donaldson}, \emph{A generalized Joyce construction for a family of nonlinear partial differential equations}, {Journal of G\"okova Geometry/Topology Conferences}, {\bf 3} (2009)
		
		\bibitem{Do2}{S. Donaldson}, \emph{Constant scalar curvature metrics on toric surfaces}, {Geometric and Functional Analysis}, {\bf 19} No. 1 (2009) {83--136}
		
		\bibitem{Etesi}{G. Etesi}, \emph{The topology of asymptotically locally flat gravitational instantons}, {Physics Letters B} {\bf 641} (2006) {461--465}
		
		%\bibitem{FYZ}{J. Fu, S. Yau, and W. Zhou}, \emph{On complete constant scalar curvature K\"ahler metrics with Poincar\'e-Mok-Yau asymptotic property}, {Communications in Analysis and Geometry}, {\bf 24} No. 3 (2-16) 521–-557
		
		\bibitem{GM}{A. Grigor’yan and J. Masamune}, \emph{Parabolicity and stochastic completeness of manifolds in terms of the Green formula}, {Journal de Mathématiques Pures et Appliquées} {\bf 9} (2013) {607--632}
		
		\bibitem{G}{V. Guillemin}, \emph{K\"ahler structures on toric varieties}, {Journal of Differential Geometry}, {\bf 40} (1994) {285--309}
		
		\bibitem{HK}{I. Holopainen and P. Koskela}, \emph{Volume growth and parabolicity}, {Proceedings of the American Mathematical Society}, {\bf 129} No. 11 (2001) {3425--3435}
		
		\bibitem{Haw77}{S. Hawking}, \emph{Gravitational instantons}. Physics Letters A \textbf{60} No. 2 (1977) 81--83.
		
		\bibitem{LT1}{P. Li and L. Tam}, \emph{Positive harmonic functions on complete manifolds with non-negative curvature outside a compact set}, {Annals of Mathematics}, {\bf 125} (1987) {171--207}
		
		\bibitem{LT2}{P. Li and L. Tam}, \emph{Harmonic functions and the structure of complete manifolds}, {Journal of Differential Geometry}, {\bf 35} (1992) {359--383}
		
		\bibitem{SymplMS}{D. McDuff and D. Salamon}. Introduction to symplectic topology. Vol. 27. Oxford University Press, 2017.
		
		%\bibitem{NW}{K. Naff and B. Weber}, \emph{Canonical metrics and ambiK\"ahler structures on 4-manifolds with $U(2)$ symmetry}, ArXiv:2111.11150
		
		\bibitem{SY}{R. Schoen and S.T. Yau}, Lectures on Differential Geometry, International Press (1994)
		
		\bibitem{TW02}{N. Trudinger amd X. Wang}, "Bernstein-J\"orgens theorem for a fourth order partial differential equation." Journal of Partial Differential Equations \textbf{15} No. 1 (2002) 78--88
		
		\bibitem{Web2}{B. Weber}, \emph{Regularity and a Liouville theorem for a class of boundary-degenerate second order equations}, Journal of Differential Equations. \textbf{281} (2021) 459--502
		
		\bibitem{Web1}{B. Weber}, \emph{Analytic classification of toric K\"ahler instanton metrics in dimension 4}, To Appear in Journal of Geometry
		
		\bibitem{Web3}{B. Weber}, \emph{Generalized K\"ahler Taub-NUT Metrics and Two Exceptional Instantons}, {Communications in Analysis and Geometry} \textbf{30} No. 7 (2022) 1575--1632
		
	\end{thebibliography}
\end{document}